\documentclass[10pt]{amsart}
\usepackage{amsmath,amssymb,amsthm,enumerate,bbm}
\usepackage{enumitem}
\usepackage{xcolor}
\usepackage{tikz, pgfplots, subfig}
\usepackage{verbatim}
\usepackage[font=footnotesize,labelfont=bf]{caption}
\allowdisplaybreaks
\numberwithin{equation}{section}

\newtheorem{theo}{Theorem}[section]
\newtheorem{coro}[theo]{Corollary}
\newtheorem{lemm}[theo]{Lemma}
\newtheorem{defi}[theo]{Definition}

\newtheorem{prop}[theo]{Proposition}
\newtheorem{obs}[theo]{Observation}

\theoremstyle{definition}
\newtheorem{exam}[theo]{Example}

\theoremstyle{remark}
\newtheorem{remark}[theo]{Remark}

\newcommand{\od}{\mathrm{d}}

\newcommand{\E}{{\mathbb{E}}}
\renewcommand{\P}{{\mathbb{P}}}
\newcommand{\R}{{\mathbb{R}}}
\newcommand{\N}{{\mathbb{N}}}

\newcommand{\cA}{\mathcal{A}}
\newcommand{\cB}{\mathcal{B}}

\newcommand{\cF}{\mathcal{F}}
\newcommand{\cP}{\mathcal{P}}

\newcommand{\ind}{\mathbbm{1}}

\newcommand{\gindset}{\mathcal{I}}
\newcommand{\Rplus}{[0, \infty)}
\newcommand{\tmax}{T_{\rm{max}}}

\newcommand{\new}[1]{#1}
\newcommand{\newer}[1]{#1}
\newcommand{\newest}[1]{#1}
\newcommand{\Newest}[1]{#1}
\newcommand{\NEWEST}[1]{#1}
\newcommand{\revised}[1]{#1}

\begin{document}

\title[Mean field SDEs with a discontinuous diffusion coefficient]{Mean field stochastic differential equations with a discontinuous diffusion coefficient}
	
\author[Jani Nyk\"{a}nen]{Jani Nyk\"{a}nen$^1$}

\maketitle

\begin{abstract}
%
%
We study $\R^d$-valued mean field stochastic differential equations with a diffusion coefficient depending on the $L_p$-norm of the process in a discontinuous way.  We show that under a strong drift there exists a unique global strong solution and consider typical cases where the existence of a global \new{solution} fails.\end{abstract}

{\small{
\vspace{1em}
{\noindent \textit{Keywords:} mean field stochastic differential equation, discontinuous diffusion coefficient, existence and non-existence of strong solutions in $L_p$  \\
\noindent \textit{Mathematics Subject Classification:} 60H10, 60H30
}
{\noindent
\footnotetext[1]{Department of Mathematics and Statistics, University of
Jyv\"askyl\"a, Finland. \\ \hspace*{1.5em}
  {\tt jani.m.nykanen@jyu.fi}}
}
}}

\tableofcontents

\section{Introduction}
%
%
We study, for $p \geq 2$, $\R^d$-valued mean field stochastic differential equations of the form
\begin{equation*}
X_t = x_0 + \int_0^t \sigma(s, X_s, \P_{X_s}, \left\Vert X_s - z \right\Vert_{L_p(\Omega)} ) \od B_s + \int_0^t b(s, X_s, \P_{X_s}) \od s, \quad t \geq 0,
\end{equation*}
where $x_0 \in L_p(\Omega, \cF, \P)$ is independent of the $d$-dimensional Brownian motion $(B_t)_{t \geq 0}$ and $z \in \R^d$ is a fixed reference point. We consider \new{an} infinite time horizon and assume that the coefficient functions $\sigma$ and $b$ are sufficiently regular in the first three parameters, but the diffusion coefficient $\sigma$ is discontinuous in the $L_p$-parameter. We assume that the discontinuity is of type
\begin{equation} \label{disc_type}
\sigma(t, x, \mu, \newest{\alpha}) = \sum_{i \in \gindset} \ind_{ \left\{ \newest{\alpha}^p \in \cA_i \right\} } \sigma_i(t, x, \mu),
\end{equation}
where $\gindset \subseteq \N$ is a nonempty index set and $\left( \cA_i \right)_{i \in \gindset}$ is a family of Borel sets on $\Rplus$.

To motivate this model, assume that $p = 2$ and consider, for example, \newer{a} phase transition in a large system of particles \new{with} weak interaction. We use our equation to model the position of the particles. We assume that the particles are initially distributed around the reference point $z$. The parameter $\left\Vert X_t - z \right\Vert_{L_2(\Omega)}$ describes the average distance from the reference point. Then, when the average distance gets larger, the density of the system decreases, and the other way around. A phase transition happens when the density of the system attains some critical value, and this affects the random movement. For example, if the density becomes too low, that is, there are \new{fewer} collisions between the particles, the random movement halts, so only the drift affects the movement of the particles. An important question is that what happens if the random movement is weak enough so that the drift can drive the particles closer \newest{or back} to the reference point, thus making the system change immediately the state again.

Mean field stochastic differential equations with irregular coefficients have been studied in various settings different \new{from} what we consider in this paper. Discontinuity in the drift coefficient under Wasserstein distance is considered in \cite{huang_discont_wasserstein}, and irregularity in the space variable in the drift coefficient is studied \cite{bauer_irregular_drift} and \cite{leobacher_discont_drift}. Existence of a solution under irregularity in the coefficients of the type
\begin{equation*}
\sigma(t, x, \mu) = \int_{\R^d} \bar{\sigma}(t, x, y) \od \mu(y), \quad b(t, x, \mu) = \int_{\R^d} \bar{b}(t, x, y) \od \mu(y),
\end{equation*}
which is sometimes called the true McKean-Vlasov case, is studied in \cite{mishura_true_mckean_vlasov} and \cite{stannat_mehri_lyapunov}. In \cite{zhang_discrete} it is proven that in this setting under sufficient assumptions one can obtain the existence of a solution even if the coefficients are discontinuous in the measure component with respect to weak convergence.

Without mean field interaction, that is, when the coefficients do not depend on the distribution, discontinuity in the diffusion coefficient is studied from different points of view for example in \cite{lejay_discontinuous_media}, \cite{kulik_deviation} and \cite{garzon_discontinuous_fractional}.

In this paper we focus on the uniqueness and existence of a strong solution. We discuss some typical \new{behavior} of equations with the type of discontinuity \newer{described in \eqref{disc_type}}. We prove that if we have a strong drift in the sense that it drives the particles away from the reference point $z$, then we always have a unique strong solution. \Newest{We show that under a sufficient condition this solution exists on the whole time interval $\Rplus$. Without assuming a strong drift a solution might not exist due to oscillating behavior, which is the topic of the remaining part of the paper.} \new{Concerning} the \newer{existence} results the main difference to \cite{zhang_discrete} is that the uniform ellipticity condition for the diffusion coefficient is not \new{required here}. \NEWEST{On the contrary, by not requiring this condition new effects are discovered.}

\NEWEST{\subsection{Structure of the paper.} We introduce the general setting and give the definition of an $L_p$-solution for our framework in Section 1.3. In Section 2 we consider the existence and uniqueness of an $L_p$-solution. Example \ref{exam_null_set_nonunique} illustrates a possible case when the uniqueness property can fail. This example is followed by Theorem \ref{theo_drift_away}, \revised{which states that} under a strong drift and some additional regularity assumptions there exists a unique solution. In Section 3 we study equations that have only finite lifetime because of an oscillating behavior, which can occur if the assumption on the strong drift is violated. 
}
\subsection{Notation}
%
%
We use the following notation:
\begin{itemize}
    \item \newer{$\left\Vert \cdot \right\Vert$ denotes the Euclidean norm on $\R^d$, or if} $A = [a_{i,j}]_{i,j=1}^d \in \R^{d \times d}$, then we use \newer{$\left\Vert \cdot \right\Vert$ to denote the Hilbert-Schmidt} norm 
    \begin{equation*}
        \left\Vert A \right\Vert = \left( \sum_{i, j = 1}^d \left| a_{i,j} \right|^2 \right)^{\frac{1}{2}}.
    \end{equation*}
    \revised{\item For $p \geq 1$ the space $L_p(\Omega, \cF, \P)$ contains all random variables $f : \Omega \to \R^d$ that satisfy $\E \left\Vert f \right\Vert^p < \infty$.}
    \new{\item If $X$ is a topological space, then $\cB(X)$ denotes the Borel $\sigma$-algebra of $X$}\newest{, generated by the open sets in $X$}.
    \revised{\item For $z \in \R^d$ we denote by $\delta_z$ the Dirac measure centered at $z \in \R^d$, that is, 
    \begin{equation*}
    \delta_z(B) = \begin{cases}
    \begin{aligned}
    & 1, \quad z \in B, \\
    & 0, \quad z \notin B
    \end{aligned}
    \end{cases}
    \end{equation*}
    for all $B \in \cB(\R^d)$.}
    \item For $p \geq 2$ we denote by $\cP_p(\R^d)$ the space of Borel probability measures on $\R^d$ with finite \newer{absolute} $p$-th moments, equipped with the $p$-Wasserstein distance \newest{
    \begin{equation} \label{wasserstein_alt_def}
    W_p(\mu, \nu)^p := \inf \E_{\P'} \left\Vert f - g \right\Vert^p,
    \end{equation}
    where the infimum is taken over all probability spaces $(\Omega', \cF', \P')$ \revised{such that there exist random variables $f, g \in L_p(\Omega', \cF', \P')$ with $\P'_f = \mu$ and $\P'_g = \nu$. The map $W_p$ is well-defined, see Remark \ref{rem_wass_existence} below.}
    }
    \item We denote the one-dimensional Lebesgue measure by $\lambda$.
    \item We denote by $\#M$ the cardinality of the set $M$.
    \item We use the convention $\N = \left\{ 1, 2, ... \right\}$.
    \item \newer{If $a, b \in \R$, then $a \wedge b = \min \left\{a, b\right\}$} \revised{and $a \vee b = \max \left\{a, b\right\}$.}
\end{itemize}

\revised{
\begin{remark} \label{rem_wass_existence}
In our definition of the $p$-Wasserstein distance given in \eqref{wasserstein_alt_def} we need to make sure that we are not taking an infimum over an empty set. Let $\mu, \nu \in \cP_p(\R^d)$. We choose the following probability space: 
\begin{equation*}
(\Omega', \cF', \P') = ( \R^d \times \R^d, \cB(\R^d \times \R^d), \mu \times \nu).
\end{equation*}
We denote by $\pi_x, \pi_y : \R^d \times \R^d \to \R^d$ the coordinate maps to the first and second coordinates, respectively, that is, $\pi_x(x, y) = x$ and $\pi_y(x, y) = y$ for all $(x, y) \in \R^d \times \R^d$. Clearly one has $\pi_x, \pi_y \in L_p(\R^d \times \R^d, \cB(\R^d \times \R^d), \mu \times \nu)$. Moreover, it holds that $\P'_{\pi_x} = \mu$ and $\P'_{\pi_y} = \nu$.

By taking the infimum over all probability spaces $(\Omega', \cF', \P')$ carrying random variables $f$ and $g$ with $\P_{f} = \mu$ and $\P_{g} = \nu$  we do not need to restrict ourselves to one atomless probability space for the definition.

\end{remark}
}
\subsection{General setting} \label{section_setting}
%
%
Assume a stochastic basis $(\Omega, \P, \cF, \left( \cF_t \right)_{t \geq 0} )$ that satisfies the usual \newer{conditions: the probability space $(\Omega, \cF, \P)$ is complete, the filtration $(\cF_t)_{t \in \Rplus}$ is right-continuous and $\cF_0$ contains all $\P$-null sets}.  Let $B = (B_t)_{t \geq 0}$ be a $d$-dimensional \newer{$(\cF_t)_{t \geq 0}$}\newest{-}Brownian motion. Let $\gindset \subseteq \N$ be a nonempty index set and let $p \geq 2$. \new{Suppose $\revised{\cF \otimes } \cB(\Rplus) \otimes \cB(\R^d) \otimes \cB(\cP_p(\R^d))$-measurable} functions
\new{
\begin{align*}
\sigma_i & : \revised{\Omega \times } [0, \infty) \times \R^d \times \cP_p(\R^d) \to \R^{d \times d}, \quad i \in \gindset, \\
b & :  \revised{\Omega \times } [0, \infty) \times \R^d \times \cP_p(\R^d) \to \R^d
\end{align*}
that are \newer{jointly} continuous in the space and measure components} for any fixed \revised{$(\omega, t) \in \Omega \times \Rplus$, and progressively measurable for any fixed $(x, \mu) \in \R^d \times \cP_p(\R^d)$}. Assume that the initial condition $x_0 \in L_p(\Omega, \cF, \P)$ is independent of $(B_t)_{t \geq 0}$, and let $z \in \R^d$ be a fixed reference point. 

Assume nonempty, pairwise disjoint Borel sets $\left( \cA_i \right)_{i \in \gindset}$ on $[0, \infty)$ that satisfy $\bigcup_{i \in \gindset} \cA_i = [0, \infty)$. We consider stochastic differential equations of the form
\begin{equation} \label{sde_general}
\begin{cases}
\begin{aligned}
& X_t = x_0 + \int_0^t \sum_{i \in \gindset} \ind_{\left\lbrace g(s) \in \cA_i \right\rbrace } \sigma_i(s, X_s, \P_{X_s}) \od B_s + \int_0^t b(s, X_s, \P_{X_s}) \od s, \\
& g(t) = \left\Vert X_t - z \right\Vert_{L_p(\Omega)}^p = \E \left\Vert X_t - z \right\Vert^p.
\end{aligned}
\end{cases}
\end{equation}
We call the function $g$ \textit{moment function.}

\new{
\begin{defi} \label{def_solution}
Recalling that $p \geq 2$, an $L_p$-solution to \eqref{sde_general} consists of a pair $\left(T, \left(X_t\right)_{t \in [0, T)} \right)$ such that the following conditions are met:
\begin{enumerate}[label=\rm{(\roman*)}]
    \item $T \in (0, \infty]$. \label{cond_sol_first}
    \item The process $X = (X_t)_{t \in [0, T)}$ has continuous sample paths and is adapted to the filtration $(\cF_t)_{t \in [0, T)}$.
    \item For any $S \in (0, T)$ one has 
    \begin{equation*}
        \E \sup_{t \in [0, S]} \left\Vert X_t - z \right\Vert^p < \infty.
    \end{equation*}
    \item For all $t \in [0, T)$ it holds that
    \begin{equation*}
    \E \left[ \int_0^t \left\Vert \sum_{i \in \gindset} \ind_{\left\{ \E \left\Vert X_u - z \right\Vert^p  \in \cA_i  \right\}} \sigma_i(u, X_u, \P_{X_u}) \right\Vert^2 \od u \right]^{\frac{p}{2}} < \infty
    \end{equation*}
    and 
    \begin{equation*}
    \E \left[ \int_0^t \left\Vert b(u, X_u, \P_{X_u}) \right\Vert \od u \right]^p < \infty.
    \end{equation*}
    \item The equation \eqref{sde_general} holds almost surely for all $t \in [0, T)$. \label{cond_sol_last}
\end{enumerate}
\newer{The pair $(T, (X_t)_{t \in [0, T)})$ is a strongly unique $L_p$-solution if for any other $L_p$-solution $(S, (Y_t)_{t \in [0, S)})$ one has $\P(X_t = Y_t) = 1$ for $t \in [0, \min \left\{ T, S \right\})$. Moreover, we say that $\tmax \in (0, \infty]$ is \newest{the} \textit{strong maximal lifetime} for the equation \eqref{sde_general} if 
\begin{enumerate}[label=\rm{(\alph*)}]
    \item there exists a strongly unique $L_p$-solution $(\tmax, (X_t)_{t \in [0, \tmax)})$, and 
    \item for any other $L_p$-solution $(S, (Y_t)_{t \in [0, S)})$ one has $S \leq \tmax$.
\end{enumerate}

}
\end{defi}


\begin{remark}
\newer{From our assumptions it follows that} the processes
\begin{equation*}
\left( \sum_{i \in \gindset} \ind_{\left\{ \E \left\Vert X_t - z \right\Vert^p  \in \cA_i  \right\}} \sigma_i(t, X_t, \P_{X_t}) \right)_{t \in [0, T)}
\end{equation*}
and
\begin{equation*}
\left(b(t, X_t, \P_{X_t}) \right)_{t \in [0, T)}
\end{equation*}
are progressively measurable.
\end{remark}

}

\newer{If it is clear from the context, then we might denote an $L_p$-solution just by $(X_t)_{t \in [0, T)}$ without mentioning the lifetime separately, especially if $T = \infty$. }

\section{Existence of \newer{an $L_p$-solution}} \label{section_existence}
%
%
We start by considering the existence and uniqueness of a strong solution. \revised{We recall that the processes $\left( \sigma_i(\cdot, t, x, \mu) \right)_{t \in \Rplus}$ and $\left( b(\cdot, t, x, \mu) \right)_{t \in \Rplus}$ are assumed to be progressively measurable for all $(x, \mu) \in \R^d \times \cP_p(\R^d)$}. For $p \geq 2$ we make the following \revised{additional} standard assumptions on the functions $\left( \sigma_i \right)_{i \in \gindset}$ and $b$:
\begin{enumerate} [label=\rm{\textbf{(S\arabic*)}}]

\item The functions $\left( \sigma_i \right)_{i \in \gindset}$ and $b$ are Lipschitz in the \revised{space} and \revised{measure} component uniformly in $t$ \revised{and $\omega$}, that is, \revised{there exist} constants $L_{\sigma_i} > 0$, $i \in \gindset$ and $L_b > 0$ such that
\begin{align*}
& \left\Vert \sigma_i(t, x, \mu) - \sigma_i(t, y, \nu) \right\Vert \leq L_{\sigma_i} \left( \left\Vert x - y \right\Vert + W_p(\mu, \nu) \right), \quad i \in \gindset, \\
& \left\Vert b(t, x, \mu) - b(t, y, \nu) \right\Vert \leq L_{b} \left( \left\Vert x - y \right\Vert + W_p(\mu, \nu) \right)
\end{align*}
for all \revised{$\omega \in \Omega$}, $t \in \Rplus$, $x, y \in \R^d$ and $\mu, \nu \in \cP_p(\R^d)$. \label{standard_assumption_1} \label{standard_assumption_first}
\item The functions $\left( \sigma_i \right)_{i \in \gindset}$ and $b$ satisfy a linear growth condition, that is, \revised{there exist} $K_{\sigma_i} > 0$, $i \in \gindset$, and $K_b > 0$ such that
\begin{align*}
& \left\Vert \sigma_i(t, x, \mu) \right\Vert \leq K_{\sigma_i} \left(1 + \left\Vert x  - z \right\Vert + W_p(\mu, \delta_z) \right), \quad i \in \gindset, \\
& \left\Vert b(t, x, \mu) \right\Vert \leq K_b \left(1 + \left\Vert x - z \right\Vert + W_p(\mu, \delta_z) \right)
\end{align*}
for all \revised{$\omega \in \Omega$}, $t \in \Rplus$ and $(x, \mu) \in \R^d \times \cP_p(\R^d)$. \label{standard_assumption_2} \label{standard_assumption_last}

\end{enumerate}\revised{
We choose these rather strong assumptions to ensure that all the special behavior is caused only by the discontinuity in the $L_p$-parameter, see the Discussion section at the end of the article.}

We recall a classical result concerning the existence and uniqueness of a strong solution. For a proof see \new{Theorem} 4.21 in \cite{mean_field_games_1}, where the result is proven for $p=2$ in finite time horizon. \newer{The proof can be generalized for any} $p \geq 2$, and the obtained solution can be extended to \newer{an} infinite time horizon in the usual way.
\begin{theo} \label{th_strong_existence}
Assume \ref{standard_assumption_first}-\ref{standard_assumption_last} \newest{and recall that $p \geq 2$}. Then the equation
\begin{align} \label{eq_sde_initial_data}
X^{(i, s, \revised{x_0})}_t = \revised{x_0} & + \int_s^t \sigma_{i} (u, X^{(i, s, \revised{x_0})}_u, \P_{X^{(i, s, \revised{x_0})}_u}) \od B_u \\ \nonumber
& + \int_s^t b (u, X^{(i, s, \revised{x_0})}_u, \P_{X^{(i, s, \revised{x_0})}_u}) \od u, \quad t \geq s,
\end{align}
has a unique global strong solution for all $(i, s, \revised{x_0}) \in \gindset \times \Rplus \new{\times} L_p(\Omega, \cF_s, \P)$ such that for all $\new{T \in [s, \infty)}$ one has
\begin{equation*}
\E \sup_{t \in [s, T]} \left\Vert X_t^{(i, s, \revised{x_0})} - z \right\Vert^p < \infty.
\end{equation*}
\end{theo}

First we observe that if sets $\cA_i$ do not have \newest{positive} mass with respect to \newest{the} image measure \newest{of the moment function}, then they can be, in some sense, ignored.

\begin{obs} \label{theo_null_set}
Assume \ref{standard_assumption_first}-\ref{standard_assumption_last}. Let \new{$Y = (Y_t)_{t \geq 0}$} be the unique strong solution to the equation
\begin{equation} \label{eq_nullset_sde_Y}
\new{
Y_t = x_0 + \int_0^t \sigma_{i_0}(s, Y_s, \P_{Y_s}) \od B_s + \int_0^t b(s, Y_s, \P_{Y_s}) \od s
}
\end{equation}
for some fixed $i_0 \in \gindset$. Let $h(t) := \E \left\Vert \newer{Y}_t - z \right\Vert^p$ for $t \geq 0$. Assume that for all $i \in \gindset \setminus \left\lbrace i_0 \right\rbrace$ it holds that $\mu( \cA_i ) = 0$, where $\mu := \lambda \circ h^{-1}$. Then $Y$ solves \eqref{sde_general}.
\end{obs}

\begin{proof}
For $t \geq 0$ we have
\new{
\begin{align*}
& x_0 + \sum_{i \in \gindset} \int_0^t \ind_{\left\lbrace h(s) \in \cA_i \right\rbrace } \sigma_i(s, Y_s, \P_{Y_s}) \od B_s + \int_0^t b(s, Y_s, \P_{Y_s}) \od s \\
& = x_0 + \int_0^t \ind_{\left\lbrace h(s) \in \cA_{i_0} \right\rbrace } \sigma_{i_0}(s, Y_s, \P_{Y_s}) \od B_s \\
& \qquad \textrm{ } + \sum_{i \neq i_0} \int_0^t \ind_{\left\lbrace h(s) \in \cA_i \right\rbrace } \sigma_i(s, Y_s, \P_{Y_s}) \od B_s + \int_0^t b(s, Y_s, \P_{Y_s}) \od s \\
& = x_0 + \int_0^t \sigma_{i_0}(s, Y_s, \P_{Y_s}) \od B_s + \int_0^t b(s, Y_s, \P_{Y_s}) \od s \\
& = Y_t
\end{align*}
}
almost surely.

\end{proof}

The following example shows that SDEs of type \eqref{sde_general} can have more than one solution:

\begin{exam} \label{exam_null_set_nonunique}
Let $p = 2$ and consider the equation
\begin{equation} \label{eq_nonunique_example}
\begin{cases}
\begin{aligned}
& X_t = 1 + \int_0^t \ind_{ \left\lbrace g(s) \notin \left\lbrace 1 + a \right\rbrace  \right\rbrace } \od B_s, \\
& g(t) = \E \left| X_t \right|^2,
\end{aligned}
\end{cases}
\end{equation}
where $a \geq 0$. Here $\cA_1 \newer{:=} [0, \infty) \setminus  \left\{ 1 + a \right\}$\newer{,} $\cA_2 \newer{:=} \left\{ 1 + a \right\}$\newer{, }$\sigma_1 \equiv 1$ and $\sigma_2 \equiv 0$. By Observation \ref{theo_null_set} the process 
\begin{equation*}
Y_t = 1 + \int_0^t 1 \od B_s = 1 + B_t
\end{equation*}
solves the equation \eqref{eq_nonunique_example}: the function $h$ is now $h(t) = 1 + t$, so $\mu = \lambda \circ h^{-1} = \left. \lambda \right|_{[1, \infty)}$, and therefore $\mu(\cA_2) = 0$. 

However, this solution is not unique. To see this, let $w > 0$ and define a process 
\begin{equation*}
{X}^w_t := 1 + B_{\min \left\lbrace a, t \right\rbrace } + \left( B_{\max \left\lbrace a + w, t \right\rbrace } - B_{a+w} \right).
\end{equation*}
The corresponding moment function is 
\begin{equation*}
{g}_w(t) = 
\begin{cases}
\begin{aligned}
& 1 + t, && \quad t < a, \\
& 1 + a, && \quad t \in [a, a + w) \\
& 1 + (t - w), && \quad t \geq a + w.
\end{aligned}
\end{cases}
\end{equation*}
Now
\begin{align*}
1 + \int_0^t \ind_{ \left\lbrace {g}_w(s) \notin \left\lbrace 1 + a \right\rbrace  \right\rbrace } \od B_s & = 1 + \int_0^t \ind_{ \left\lbrace s \in [0, a) \cup [a + w, \infty)  \right\rbrace } \od B_s \\
& = 1 + \int_0^{t \wedge a}  \od B_s + \int_{a + w}^{ t \vee (a + w)} \od B_s \\
& = {X}_t^{\newer{w}}.
\end{align*}
Therefore the equation \eqref{eq_nonunique_example} has infinitely many solutions.

\end{exam}

Next we show that under some additional assumptions on the sets $\cA_i$, if the drift is strong in the sense that it is constantly driving the particles away from the reference point $z$, then there exists a \newer{strongly unique $L_p$-solution}\newest{.}

\begin{theo} \label{theo_drift_away}
Assume \ref{standard_assumption_first}-\ref{standard_assumption_last}. Suppose that $\gindset = \N$. Let $(y_i)_{i=0}^\infty \subset \Rplus$ such that $y_0 = 0$, $y_{i-1} < y_{i}$ for all $i \in \gindset$ and $y_i \to \infty$ as $i \to \infty$. Assume that 
\begin{enumerate}[label = \rm{(\roman*)}]
    \item for each $i \in \gindset$ it holds that $(y_{i-1}, y_i) \subseteq \cA_i \subseteq [y_{i-1}, y_i]$, \label{th_drift_first}
    \item $\P (x_0 \neq z) > 0$, and
    \item $\left\langle x - z, b(t, x, \mu) \right\rangle \geq 0$ for all $(t, x, \mu) \in [0, \infty) \times \R^d \times \cP_p(\R^d)$ and the inequality is strict when $x \neq z$. \label{as_strong_drift}  \label{th_drift_last}
\end{enumerate}
\new{
Then the equation \eqref{sde_general} has a \newest{strongly unique} $L_p$-solution $(\tmax, X_{t \in [0, \tmax)})$, where $\tmax \in (0, \infty]$ is the maximal lifetime. Moreover, if 
\begin{equation} \label{suff_cond_tmax}
\sum_{k=2}^\infty \frac{1}{K_b + K_{\sigma_k}^2} \log \left( \frac{y_k}{y_{k-1}} \right) = \infty,
\end{equation}
then $\tmax = \infty$.
}

\end{theo}

\begin{proof}

\underline{\textbf{Step A:}} Fix $(i, s) \in \gindset \times \Rplus$ and $\xi \in L_p(\Omega, \cF_s, \P)$ such that $\P(\xi \neq z) > 0$. Let $Y = X^{(i, s, \xi)}$, where the process $\left( X^{(i, s, \xi)}_t \right)_{t \geq s}$ is obtained in Theorem \ref{th_strong_existence}. In other words, $Y = (Y_t)_{t \geq s}$ is a unique strong solution to the equation
\begin{equation*}
Y_t = \xi + \int_s^t \sigma_i(u, Y_u, \P_{Y_u}) \od B_u + \int_s^t b(u, Y_u, \P_{Y_u}) \od u, \quad t \geq s,
\end{equation*}
such that \newer{
\begin{equation} \label{finite_moment_Y}
\E \sup_{t \in [s, T]} \left\Vert Y_t - z \right\Vert^p < \infty
\end{equation} 
}for all $T \in (s, \infty)$. Then, by applying It\^{o}'s formula to the function $x \mapsto \left\Vert x - z \right\Vert^p$ and taking the expectation, we have
\begin{align} \label{ito1_drift_away}
\E \left\Vert Y_t - z \right\Vert^p & = \E \left\Vert \xi - z \right\Vert^p \\ \nonumber
& + p  \int_s^t \E \left[  \left\Vert Y_u - z \right\Vert^{p-2} \left\langle Y_u - z, b(u, Y_u, \P_{Y_u }) \right\rangle \right] \od u \\ \nonumber
& + \frac{p}{2} \int_s^t \E \left[  \left\Vert Y_u - z \right\Vert^{p-2} \left\Vert \sigma_{i} (u, Y_u, \P_{Y_u }) \right\Vert^2 \right]  \od u \\ \nonumber
& + \frac{p(p-2)}{2} \int_s^t \E \left[ \left\Vert Y_u - z \right\Vert^{p-4} R_{(u, i)} \right]  \od u,
\end{align}
where 
\begin{equation} \label{remainder_process}
R_{(u, i)} := \left\langle \sigma_{i} \sigma_{i}^{\top} (u, Y_u, \P_{Y_u}) (Y_u  - z), Y_u  - z \right\rangle \geq 0.
\end{equation}
\newer{One notices that the stochastic integral term does not appear in \eqref{ito1_drift_away} since its expectation is zero, see Remark \ref{rem_martingale}.}

Here we already see that the function $[s, \infty) \ni t \mapsto \E \left\Vert Y_t - z \right\Vert^p$ is non-decreasing. Then $\E \left\Vert Y_t - z \right\Vert^p \geq \E \left\Vert \xi - z \right\Vert^p > 0$, so we have $\P \left( Y_t \neq z \right) > 0$ for all $t \geq s$. Therefore
\begin{equation*}
\left\langle Y_t - z, b(t, Y_t, \P_{Y_t }) \right\rangle > 0
\end{equation*}
with positive probability for all $t \geq s$, and in particular
\begin{equation*}
\E \left[  \left\Vert Y_t - z \right\Vert^{p-2} \left\langle Y_t - z, b(t, Y_t, \P_{Y_t }) \right\rangle \right] > 0
\end{equation*}
for $t \geq s$. Hence the function $[s, \infty) \ni t \mapsto \E \left\Vert Y_t - z \right\Vert^p$ is strictly increasing.

\underline{\textbf{Step B:}} Without loss of generality, we can assume that $\E \left\Vert x_0 - z \right\Vert^p \in (0, y_1]$. Let $T_0 := 0$. Let $X^1 \newer{:}= \left(X^{(1, T_0, x_0)}\right)_{t \geq 0}$ and define 
\begin{equation*}
T_1 := \inf \left\{ t \geq 0 \textrm{ } \bigg| \hspace{0.30em} \E \left\Vert X_t^1 - z \right\Vert^p = y_1 \right\}.
\end{equation*}
If $T_1 = \infty$, then we let $T := \infty$ and stop. Otherwise, let $X^2 = \left(X^{(2, T_1, X^1_{T_1})}_t\right)_{t \geq T_1}$ and define 
\begin{equation*}
T_2 := \inf \left\{ t \geq T_1 \textrm{ } \bigg| \hspace{0.30em}  \E \left\Vert X_t^2 - z \right\Vert^p = y_2 \right\}.
\end{equation*}
Again, if $T_2 = \infty$, then we stop and let $T := \infty$, otherwise we continue.

For arbitrary $i \in \gindset$ \newest{with} $i > 1$ \newest{and $T_{i-1} < \infty$}, let $X^i = \left(X^{(i, T_{i-1}, X^{i-1}_{T_{i-1}})}_t\right)_{t \geq T_{i-1}}$ and define 
\begin{equation*}
T_i := \inf \left\{ t \geq T_{i-1} \textrm{ } \bigg| \hspace{0.30em}  \E \left\Vert X_t^i - z \right\Vert^p = y_i \right\}.
\end{equation*}
In the case $T_i = \infty$ we stop and let $T := \infty$. Otherwise we continue and let 
\begin{equation*}
T := \lim_{i \to \infty} T_i \in (0, \infty].
\end{equation*}
On each interval $[T_{i-1}, T_i]$ the process $(X^i_t)_{t \in [T_{i-1}, T_i]}$ is a strong solution to \eqref{sde_general} because the coefficient $\sigma$ used in $X^i$ is obtained from $\sigma_i$. \new{For all $i \in \gindset$ the process $X^{i+1}$ starts from the \newer{"}ending point\newer{"} of the process $X^{i}$, that is, $X_{T_i}^{i+1} = X_{T_i}^{i}$}, so we can construct a process $X = (X_t)_{t \in [0, T)}$ such that $X_t = X^i_t$ for all $t \in [T_{i-1}, T_i]$, $i \in \gindset$. Therefore $X$ is \newest{an $L_p$-}solution to \eqref{sde_general} on the interval $[0, T)$.

Let $S \in (0,T)$. Then there \new{are} only finitely many $T_i$ such that $T_i \leq S$. Since for each $i \in \gindset$ one has 
\begin{equation*}
\E \sup_{t \in [T_{i-1}, T_i]} \left\Vert X_t^i - z \right\Vert^p < \infty,
\end{equation*}
we obtain that 
\begin{equation*}
\E \sup_{t \in [0, S]} \left\Vert X_t - z \right\Vert^p < \infty.
\end{equation*}
\new{Moreover, if $T < \infty$, then \newer{we must have}
\begin{equation} \label{infty_moment}
\lim_{S \uparrow T} \E \sup_{t \in [0, S]} \left\Vert X_t - z \right\Vert^p = \infty,
\end{equation}
}
\newer{which follows from the assumption that $y_k \uparrow \infty$ as $k \to \infty$.}

\underline{\textbf{Step C:}} \newer{To prove the uniqueness of \newest{the} solution, let us consider two $L_p$-solutions to \eqref{sde_general}, denoted by $(T, (X_t)_{t \in [0, T)})$ and $(S, (Y_t)_{t \in [0, S)})$. Let $T_1, T_2, ...$ and $S_1, S_2, ...$ be the time points obtained in Step B for $X = (X_t)_{t \in [0, T)}$ and $Y = (Y_t)_{t \in [0, S)}$\newest{,} respectively.

Let $R_1 := \min \left\{ T_1, S_1 \right\}$. Since both processes $X$ and $Y$ have the same initial value, they solve 
\begin{equation} \label{stepC_sde1}
Z_t = x_0 + \int_0^t \sigma_1(s, Z_s, \P_{Z_s}) \od B_s + \int_0^t b(s, Z_s, \P_{Z_s}) \od s 
\end{equation}
for $t \in [0, R_1)$. By Theorem \ref{th_strong_existence} the equation \eqref{stepC_sde1} has a unique strong solution, hence $\P(X_t = Y_t) = 1$ for $t \in [0, R_1)$. In particular 
\begin{equation*}
\E \left\Vert X_t - z \right\Vert^p = \E \left\Vert Y_t - z \right\Vert^p
\end{equation*}
for $t \in [0, R_1)$, so $T_1 = S_1$.

\Newest{
Assume that $R_1 < \infty$ and let $R_2 := \min \left\{ T_2, S_2 \right\}$. By the continuity of the functions $t \mapsto \E \left\Vert X_t - z \right\Vert^p$ and $t \mapsto \E \left\Vert Y_t - z \right\Vert^p$ we get $R_2 > R_1$. Moreover, by Step A these functions are strictly increasing on $(R_1, R_2)$ so that $\E \left\Vert Y_t - z \right\Vert^p \in \cA_2$ and $\E \left\Vert X_t - z \right\Vert^p \in \cA_2$ for all $t \in (R_1, R_2)$. This implies that both $X$ and $Y$ solve the same SDE 
\begin{equation*}
Z_t = X_{R_1} + \int_{R_1}^t \sigma_2(s, Z_s, \P_{Z_s}) \od B_s + \int_{R_1}^t b(s, Z_s, \P_{Z_s}) \od s 
\end{equation*}
on $(R_1, R_2)$. By Theorem \ref{th_strong_existence} we have $\P(X_t = Y_t) = 1$ for $t \in [R_1, R_2)$ and therefore $T_2 = S_2$.}

\Newest{Assuming that $R_2 < \infty$,} we repeat the same argument inductively to get that \Newest{either} $T_i = S_i$ for all $i \in \N$ with $T_i < \infty$, \Newest{or $S_i = T_i = \infty$ for some $i \in \N$}, eventually obtaining that $S = T$ and $\P(X_t = Y_t) = 1$ for all $t \in [0, T)$.
}

\new{Since we have now obtained the uniqueness of the solution, it is justified to consider \newer{the} maximality of \newest{the} lifetime. If $T = \infty$, it is clear that $T = \tmax$. In the case $T < \infty$ we use \eqref{infty_moment} to conclude that $T$ is the maximal lifetime.}

\underline{\textbf{Step D:}} \newer{\newest{It remains} to show that if the condition \eqref{suff_cond_tmax} holds, then $\tmax = \infty$.} \new {Let 
\begin{equation*}
n_0 := \min \left\{ n \in \N \mid y_n \geq \newer{2} \right\}
\end{equation*}
and fix any $N > n_0$. We can assume that $T_N < \infty$, \newer{because otherwise $\tmax = T_N = \infty$ and we are done}. Let $X^N$ be the process defined in Step B. Since 
\begin{equation*}
\E \sup_{t \in [T_{N-1}, T_N]} \left\Vert X_t^N - z \right\Vert^p < \infty
\end{equation*} 
we can apply It\^{o}'s formula to the function $x \mapsto \left\Vert x \right\Vert^p$ and take the expectation to obtain that
\begin{align} \label{eq_stepD_main}
\E \left\Vert X_{t}^N - z \right\Vert^p & = y_{N-1} \\ \nonumber
& + p  \int_{T_{N-1}}^t \E \left[  \left\Vert X_u^N - z \right\Vert^{p-2} \left\langle X_u^N - z, b(u, X_u^N, \P_{X_u^N }) \right\rangle \right] \od u \\ \nonumber
& + \frac{p}{2} \int_{T_{N-1}}^t \E \left[  \left\Vert X_u^N - z \right\Vert^{p-2} \left\Vert \sigma_N (u, X_u^N, \P_{X_u^N }) \right\Vert^2 \right]  \od u \\ \nonumber
& + \frac{p(p-2)}{2} \int_{T_{N-1}}^t \E \left[ \left\Vert X_u^N - z \right\Vert^{p-4} R_{(u, N)} \right]  \od u, \nonumber
\end{align}
for $t \in [T_{N-1}, T_N)$, where \newer{again}
\begin{equation*}
R_{(t, N)} := \left\langle \sigma_N \sigma_N^{\top} (t, X_t^N, \P_{X_t^N}) (X_t^N  - z), X_t^N  - z \right\rangle.
\end{equation*}
\newest{Similarly as in Step A, the stochastic integral term does not appear in \eqref{eq_stepD_main}, see Remark \ref{rem_martingale}.}

\newer{We recall that if $f, g \in L^p(\Omega, \cF, \P)$, then 
\begin{equation} \label{wass_estimate}
W_p(\P_f, \P_g)^p \leq \E \left\Vert f - g \right\Vert^p,
\end{equation}
which can be seen from our definition for $p$-Wasserstein distance in \eqref{wasserstein_alt_def}. Using this observation,} the Cauchy-Schwarz inequality, the linear growth condition \ref{standard_assumption_2} and H{\"o}lder's inequality we obtain that 
\begin{align*}
& \E \left\Vert X_t^N - z \right\Vert^{p-2} \left\langle X_t^N - z, b(t, X_t^N, \P_{X_t^N}) \right\rangle \\
& \leq \E \left[ \left\Vert X_t^N - z \right\Vert^{p-1} \left\Vert b(t, X_t^N, \P_{X_t^N}) \right\Vert \right] \\
& \leq K_b \E \left[ \left\Vert X_t^N - z \right\Vert^{p-1} \left( 1 + \left\Vert X_t^N - z \right\Vert + W_p(\P_{X_t^N}, \delta_z) \right) \right] \\
& \leq K_b \left[ \E \left\Vert X_t^N - z \right\Vert^{p-1} + \E \left\Vert X_t^N - z \right\Vert^{p} + \E \left\Vert X_t^N - z \right\Vert^{p-1} \left( \E \left\Vert X_t^N - z \right\Vert^{p} \right)^{\frac{1}{p}} \right] \\
& \leq K_b \left[ \left( \E \left\Vert X_t^N - z \right\Vert^{p} \right)^{\frac{p-1}{p}} + 2 \E \left\Vert X_t^N - z \right\Vert^{p}  \right] \\
& \leq 3 K_b \E \left\Vert X_t^N - z \right\Vert^{p}.
\end{align*}
In the last inequality we used the fact that $N > n_0$, which implies that for any $t \in [T_{N-1}, T_N)$ one has $\E \left\Vert X_t^N - z \right\Vert^{p} \geq y_{N-1} \geq 1$ , hence 
\begin{equation*}
\left( \E \left\Vert X_t^N - z \right\Vert^{p} \right)^{\alpha} \leq \E \left\Vert X_t^N - z \right\Vert^{p} 
\end{equation*}
for any exponent $\alpha \in (0, 1)$.

Similarly, we use \newer{the inequality \eqref{wass_estimate}}, the assumption \ref{standard_assumption_2} and H{\"o}lder's inequality to see that 
\begin{align*}
& \E \left[ \left\Vert X_t^N - z \right\Vert^{p-2} \left\Vert \sigma_N(t, X_t^N, \P_{X_t^N}) \right\Vert^{2} \right] \\
& \leq  K_{\sigma_N}^2 \E \left[ \left\Vert X_t^N - z \right\Vert^{p-2} \left( 1 + \left\Vert X_t^N - z \right\Vert + W_p(\P_{X_t^N}, \delta_z) \right)^{2} \right] \\
& \leq 3 K_{\sigma_N}^2 \left[\E \left\Vert X_t^N - z \right\Vert^{p-2} + 2 \E \left\Vert X_t^N - z \right\Vert^{p}  \right] \\
& \revised{ \leq 3 K_{\sigma_N}^2 \left[ \left(\E  \left\Vert X_t^N - z \right\Vert^{p} \right)^{\frac{p-2}{p}} + 2 \E \left\Vert X_t^N - z \right\Vert^{p}  \right] } \\
& \leq 9 K_{\sigma_N}^2 \E \left\Vert X_t^N - z \right\Vert^{p}.
\end{align*}

By the Cauchy-Schwarz inequality and the assumption \ref{standard_assumption_2} one gets 
\begin{align*}
R_{(t, N)} & \leq \left\Vert \sigma_N(t, X_t^N, \P_{X_t^N}) \right\Vert \left\Vert \sigma_N^{\top}(t, X_t^N, \P_{X_t^N}) \right\Vert \left\Vert X_t^N - z \right\Vert^2 \\
& \leq K_{\sigma_N}^2 \left( 1 + \left\Vert X_t^N - z \right\Vert + W_p(\P_{X_t^N}, \delta_z) \right)^2 \left\Vert X_t^N - z \right\Vert^2.
\end{align*}
Using the estimate above, H\"{o}lder's inequality \newer{and the inequality \eqref{wass_estimate}} gives us
\begin{align*}
& \E \left\Vert X_t^N - z \right\Vert^{p-4} R_{(t, N)} \\
& \leq K_{\sigma_N}^2 \E \left[ \left\Vert X_t^N - z \right\Vert^{p-2} \left( 1 + \left\Vert X_t^N - z \right\Vert + W_p(\P_{X_t^N}, \delta_z) \right)^2  \right] \\
& \leq 9 K_{\sigma_N}^2 \E \left\Vert X_t^N - z \right\Vert^{p}.
\end{align*}

Combining all the estimates above yields
\begin{equation*}
\E \left\Vert X_{t}^N - z \right\Vert^p \leq y_{N-1}  + C\newer{(p)} \left(K_b + K_{\sigma_N}^2 \right) \int_{T_{N-1}}^t \E \left\Vert X_s^N - z \right\Vert^{p} \od s,
\end{equation*}
where $C\newer{(p)} > 0$ is a constant \newer{depending only on $p$}. Thus by Gronwall's inequality
\begin{equation*}
\E \left\Vert X_{t}^N - z \right\Vert^p \leq y_{N-1} \exp( C\newer{(p)} \left(K_b + K_{\sigma_N}^2\right) (t - T_{N-1})).
\end{equation*}
We solve the equation
\begin{equation*}
y_{N-1} \exp( C\newer{(p)} \left(K_b + K_{\sigma_N}^2\right) (\hat{T} - T_{N-1})) = y_N
\end{equation*}
with respect to $\hat{T}$ to obtain that 
\begin{equation*}
\hat{T} = \frac{1}{C\newer{(p)}} \frac{1}{K_b + K_{\sigma_N}^2} \log \left( \frac{y_N}{y_{N-1}} \right) + T_{N-1}.
\end{equation*}
Since $T_\newer{N} \geq \hat{T}$ we get
\begin{align*}
T_N & \geq \frac{1}{C\newer{(p)}} \frac{1}{K_b + K_{\sigma_N}^2} \log \left( \frac{y_N}{y_{N-1}} \right) + T_{N-1} \\
& \geq \frac{1}{C\newer{(p)}} \sum_{n = n_0 + 1}^N \revised{\frac{1}{K_b + K_{\sigma_n}^2}} \log \left( \frac{y_n}{y_{n-1}} \right).
\end{align*}
Now either there \newest{exists an $N \in \N$} such that $T_{N} = \infty$, or \newest{otherwise}
\begin{equation*}
\tmax = \lim_{N \to \infty} T_N \geq \frac{1}{C\newer{(p)}} \sum_{n = n_0 + 1}^\infty \revised{\frac{1}{K_b + K_{\sigma_n}^2}} \log \left( \frac{y_n}{y_{n-1}} \right) = \infty
\end{equation*}
by assumption \eqref{suff_cond_tmax}.

}

\end{proof}

\begin{remark} \label{rem_martingale}
\newest{
Let
\begin{align*}
M = \left( M_t \right)_{t \geq s} := & \left( \int_{s}^{t} \left\Vert Y_u - z \right\Vert^{p-2} \left\langle \revised{\sigma^{\top}_i(u, Y_u, \P_{Y_u})} (Y_u - z), \od B_u \right\rangle   \right)_{t \geq s} \\
= & \left( \sum_{j=1}^d \int_s^{t}  \left\Vert Y_u - z \right\Vert^{p-2} \left[ \revised{\sigma_i^{\top} (u, Y_u, \P_{Y_u})} (Y_u - z) \right]_j \od B_u^j \right)_{t \geq s},
\end{align*}
where $\left[ \sigma_i(u, Y_u, \P_{Y_u}) (Y_u - z) \right]_j$ denotes the $j$:th component of the corresponding $d$-dimensional vector. The process $M$ is a martingale and, in particular, $\E M_t = 0$ for all $t \geq s$. 

To see this, we first observe that since the process $(Y_t)_{t \geq s}$ is adapted and has continuous sample paths, the processes inside the stochastic integrals are progressively measurable. Moreover,
\begin{equation*}
\P \left( \int_s^T \left\Vert Y_u - z \right\Vert^{2(p-2)} \left[ \sigma_i(u, Y_u, \P_{Y_u}) (Y_u - z) \right]_j^2 \od u < \infty \right) = 1
\end{equation*}
for all $j \in \left\{1, ..., d\right\}$ and any $T \in (s, \infty)$, so the stochastic integrals are well-defined continuous local martingales. By the Burkholder-Davis-Gundy inequality it suffices to show that $\E \sqrt{\left\langle M \right\rangle_t} < \infty$ for all $t \geq s$.

We notice that
\begin{align*}
\left\langle M \right\rangle_t \leq \int_s^t \left\Vert Y_u - z \right\Vert^{2(p-1)} \left\Vert \sigma_i(u, Y_u, \P_{Y_u}) \right\Vert^2 \od u
\end{align*}
for all $t \geq s$, so we obtain the estimate
\begin{align*}
\E \sqrt{ \left\langle M \right\rangle_t } & \leq \E \sqrt{ \int_s^t \sup_{r \in [s, t]} \left[ \left\Vert Y_r - z \right\Vert^{2(p-1)} \left\Vert \sigma_i(r, Y_r, \P_{Y_r}) \right\Vert^2 \right] \od u } \\
& \leq (\sqrt{t - s}) \E \sup_{r \in [s, t]} \left[ \left\Vert Y_r - z \right\Vert^{p-1} \left\Vert \sigma_i(r, Y_r, \P_{Y_r}) \right\Vert \right]
\end{align*}
for $t \geq s$.

Using the linear growth assumption \ref{standard_assumption_2} for $\sigma_i$ we get 
\begin{align*}
& \E \sup_{r \in [s, t]} \left[ \left\Vert Y_r - z \right\Vert^{p-1} \left\Vert \sigma_i(r, Y_r, \P_{Y_r}) \right\Vert \right] \\
& \leq \revised{K_{\sigma_i}} \E \sup_{r \in [s, t]} \left[ \left\Vert Y_r - z \right\Vert^{p-1} \left[ 1 + \left\Vert  Y_r - z \right\Vert + W_p(\P_{Y_r}, \delta_z)  \right] \right] \\
& \leq \revised{K_{\sigma_i}}  \left[ \E \sup_{r \in [s, t]} \left\Vert Y_r - z \right\Vert^{p-1} + \E \sup_{r \in [s, t]} \left\Vert Y_r - z \right\Vert^{p} + \E \sup_{r \in [s, t]} \left[ \left\Vert Y_r - z \right\Vert^{p-1} \left( \E \left\Vert Y_r - z \right\Vert^p \right)^{\frac{1}{p}} \right] \right],
\end{align*}
where we used the fact that $\revised{W_p(\P_f, \delta_z)} = \left( \E \left\Vert f - z \right\Vert^p \right)^{\frac{1}{p}}$ for any random variable $f \in L_p(\Omega, \cF, \P)$. We recall that 
\begin{equation*}
\E \sup_{r \in [s, t]} \left\Vert Y_r - z \right\Vert^{p} < \infty
\end{equation*}
by Theorem \ref{th_strong_existence}. By H{\"o}lder's inequality we have 
\begin{equation*}
\E \sup_{r \in [s, t]} \left\Vert Y_r - z \right\Vert^{p-1} \leq \left( \E \sup_{r \in [s, t]} \left\Vert Y_r - z \right\Vert^{p} \right)^{\frac{p-1}{p}} < \infty.
\end{equation*}
Similarly we obtain 
\begin{align*}
& \E \sup_{r \in [s, t]} \left[ \left\Vert Y_r - z \right\Vert^{p-1} \left( \E \left\Vert Y_r - z \right\Vert^p \right)^{\frac{1}{p}} \right] \\
& \leq  \left( \E \sup_{r \in [s, t]}  \left\Vert Y_r - z \right\Vert^{p} \right)^{\frac{p-1}{p}} \left( \E \sup_{r \in [s, t]} \left\Vert Y_r - z \right\Vert^p \right)^{\frac{1}{p}} \\
& < \infty.
\end{align*}

We conclude that $E \sqrt{ \left\langle M \right\rangle_t } < \infty$ for all $t \geq s$, hence the process $M$ is a martingale. 
}

\end{remark}

\new{

As a consequence of the Theorem \ref{theo_drift_away} \newest{we will show the following:} If the linear growth constants $\left( K_{\sigma_i} \right)_{i \in \N}$ are uniformly bounded, then we always have $\tmax = \infty$, and in particular this does not depend on the choice of $(y_k)_{k \in \N}$ \newer{as long as $y_k \uparrow \infty$.}

\begin{coro}
Suppose that the assumptions \ref{th_drift_first}-\ref{th_drift_last} in Theorem \ref{theo_drift_away} hold. If there is a constant $K > 0$ such that
\begin{equation*}
\sup_{i \in \N} K_{\sigma_i} \leq K < \infty,
\end{equation*}
then there exists a \newest{strongly unique} $L_p$-solution with a maximal lifetime $\tmax = \infty$.
\end{coro}

\begin{proof}
We only need to show that $\tmax = \infty$. \revised{Since 
\begin{align*}
\sum_{k=2}^N \log \left( \frac{y_k}{y_{k-1}} \right) & = \sum_{k=2}^N \left[ \log(y_k) - \log(y_{k-1}) \right] \\
& = \log(y_N) - \log(y_1) \to \infty
\end{align*}
as $N \to \infty$, it holds that
\begin{equation*}
\sum_{k = 2}^\infty \frac{1}{K_b + K_{\sigma_k}^2} \log \left( \frac{y_k}{y_{k-1}} \right) \geq \frac{1}{K_b + K^2} \sum_{k = 2}^\infty \log \left( \frac{y_k}{y_{k-1}} \right) = \infty.
\end{equation*}
We conclude that $\tmax = \infty$ by Theorem \ref{theo_drift_away} }
\end{proof}

In the following example we consider a simple equation with unbounded linear growth constants, that is, $\sup_{i \in \N} K_{\sigma_i} = \infty$. By varying parameters in this equation we can find examples of cases when $\tmax < \infty$ \newest{and} $\tmax = \infty$ (even if \eqref{suff_cond_tmax} is not satisfied).

\begin{exam} \label{ex_converging_T}
Consider the equation
\begin{equation} \label{eq_example_bounded_lips}
\begin{cases}
\begin{aligned}
X_t & = x_0 + \sum_{n = 1}^\infty \int_0^t \ind_{ \left\{ g(s) \in [y_{n-1}, y_n) \right\} } n^{\alpha} \od B_s + \int_0^t X_s \od s, \\
g(t) & = \E \left| X_t \right|^2
\end{aligned}
\end{cases}
\end{equation}
for some exponent $\alpha > 0$. \newest{In this example $z = 0$ and $d = 1$}. We have $K_{\sigma_n} = n^{\alpha}$ for all $n \in \N$ and $K_b = 1$. For simplicity we suppose that $\E \left| x_0 \right|^2 \in (0, y_1)$. By Theorem \ref{theo_drift_away} the equation \eqref{eq_example_bounded_lips} has a \newer{strongly unique $L_p$-solution $(\tmax, (X_t)_{t \in [0, \tmax)})$, where $\tmax \in (0, \infty]$}.

Let us fix $n \geq 2$. Let $X^n$ and $T_n$ be like in the proof of Theorem \ref{theo_drift_away}. Applying It{\^o}'s formula one notices that the function $t \mapsto \E \left| X_t^n \right|^2$ solves the integral equation
\begin{equation} \label{ex_explosion_ode}
\E \left| X_t^n \right|^2 = y_{n-1} + 2 \int_{\new{T_{n-1}}}^t \E \left| X_s^n \right|^2 \od s + n^{2 \alpha}(t - T_{n-1}), \quad t \in [T_{n-1}, T_n).
\end{equation}
\newest{The} solution to \eqref{ex_explosion_ode} is given by 
\begin{equation*}
\E \left| X_t^n \right|^2 = \left( \frac{1}{2}n^{2\alpha} + y_{n-1} \right) \exp \left( 2 (t - T_{n-1}) \right) - \frac{1}{2}n^{2 \alpha}, \quad t \in [T_{n-1}, T_n).
\end{equation*}
We obtain $T_n$ by solving the equation $\E \left| X_{T_n}^n \right|^2 = y_n$, thus\newest{, for $n \geq 2$,}
\begin{align*}
T_n & = \frac{1}{2} \log \left( \frac{y_n + \frac{1}{2} n^{2\alpha} }{y_{n-1} + \frac{1}{2} n^{2 \alpha} } \right) + T_{n-1} \\
& = \frac{1}{2} \sum_{k=2}^n \log \left( \frac{y_k + \frac{1}{2} k^{2\alpha} }{y_{k-1} + \frac{1}{2} k^{2 \alpha} } \right) + T_1.
\end{align*}
By choosing different $\alpha$ and $y_k$ we can give examples of cases when $\tmax = \infty$ or $\tmax < \infty$. It is sufficient to consider convergence of the series 
\begin{equation} \label{series1}
\sum_{k=2}^{\infty} \log \left( \frac{y_k + \frac{1}{2} k^{2\alpha} }{y_{k-1} + \frac{1}{2} k^{2 \alpha} } \right)\newer{.}
\end{equation}

First, let us take $y_k = k$ for all $k \in \N$. Then one can write 
\begin{equation*}
T_n = \frac{1}{2} \sum_{k=2}^n \log \left( \frac{2k + k^{2\alpha} }{2k - 2 + k^{2 \alpha} } \right) + T_1.
\end{equation*}
If $\alpha \in (0, \frac{1}{2}]$, then \newest{
\begin{equation*}
\sum_{k=2}^n \log \left( \frac{2k + k^{2\alpha} }{2k - 2 + k^{2 \alpha} } \right) = \sum_{k=2}^n \log \left(1 + \frac{2}{2k - 2 + k^{2 \alpha} } \right).
\end{equation*}
For large enough $k \in \N$ the term $\log \left(1 + \frac{2}{2k - 2 + k^{2 \alpha} } \right)$ is comparable to $\frac{2}{2k - 2 + k^{2 \alpha}}$. Since
\begin{equation*}
\frac{2}{2k - 2 + k^{2 \alpha}} \geq \frac{2}{3k-2},
\end{equation*}
}
the series \eqref{series1} diverges, thus $\tmax = \infty$. However, we still have 
\begin{equation*}
\sum_{k=2}^{\infty} \frac{1}{1 + k^{2\alpha}} \log \left( \frac{k}{k-1} \right) < \infty,
\end{equation*}
which shows that \eqref{suff_cond_tmax} is not a necessary condition to obtain that $\tmax = \infty$.

Next let us consider the case $\newer{\alpha \in (\frac{1}{2}, 1]}$. If we take $y_k = k$ for all $k \geq 1$, then the series \eqref{series1} converges, hence $\tmax < \infty$, but by taking $y_k = (k!)^k$ for $k \geq 1$ we obtain the series \revised{
\begin{equation*}
 \sum_{k=2}^{\infty} \log \left( \frac{2 (k!)^k + k^{2\alpha} }{2 ((k-1)!)^{k-1} + k^{2\alpha} } \right).
\end{equation*}}
In this case it is easier to consider the condition \eqref{suff_cond_tmax}, which \revised{is now
\begin{align*}
\sum_{k=2}^{\infty} \frac{1}{1+ k^{2\alpha}} \log\left( \frac{(k!)^k }{((k-1)!)^{k-1}}  \right) & = \sum_{k=2}^{\infty} \frac{\log(k!) + (k-1) \log(k) }{1 + k^{2\alpha}} \\
& \geq \sum_{k=2}^{\infty} \frac{ (k-1) \log(k) }{1 + k^{2}} \\
& \geq \sum_{k=2}^{\infty} \frac{\log(k) }{2k} - \sum_{k=2}^{\infty} \frac{ \log(k) }{1 + k^{2}} = \infty,
\end{align*}
}so by Theorem \ref{theo_drift_away} we conclude that $\tmax = \infty$. This shows that the choice of $y_k$ affects the convergence of the series \eqref{suff_cond_tmax}.

\end{exam}
}

\section{Finite lifetime because of oscillation} \label{section_nonexistence}
%
%
In this section we study equations with a finite maximal lifetime due to an oscillating behavior.

We assume that $\gindset = \left\lbrace 1, 2 \right\rbrace$ and that there exists a $y > 0$ such that either $\cA_1 = [0, y)$ or $\cA_1 = [0, y]$, and $\cA_2 = [0, \infty) \setminus \cA_1$. Now the equation \eqref{sde_general} can be written as 
\begin{equation} \label{sde_general_single_level}
\begin{cases}
\begin{aligned}
& X_t = x_0 + \sum_{i =1 }^2 \int_0^t \ind_{\left\lbrace g(s) \in \cA_i \right\rbrace } \sigma_i(s, X_s, \P_{X_s}) \od B_s + \int_0^t b(s, X_s, \P_{X_s}) \od s, \\
& g(t) = \E \left\Vert X_t - z \right\Vert^p.
\end{aligned}
\end{cases}
\end{equation}
We observe that if $\E \left\Vert x_0 - z \right\Vert^p \neq y$, then we can always find a \newest{strongly} unique $L_p$-solution $(T, (X_t)_{t \in [0, T)})$ such that $g(t) \neq y$ for all $t \in [0, T)$:

\begin{obs} \label{obs_some_solution}
Suppose that $\E \left\Vert x_0 - z \right\Vert^p \neq y$ \Newest{and \ref{standard_assumption_first}-\ref{standard_assumption_last} hold}. Then there exists a strongly unique $L_p$-solution $(T_0, (X_t)_{t \in [0, T_0)})$ to \eqref{sde_general_single_level} such that 
\begin{equation*}
T_0 := \inf \left\{ t \geq 0 \hspace{0.25em} \Big| \hspace{0.25em} \lim_{s \uparrow t} g(s) = y \right\} \in (0, \infty].
\end{equation*}
If $T_0 < \infty$, then the solution can be extended to $[0, T_0]$ \revised{and it holds that 
\begin{equation*}
T_0 = \inf \left\{ t \geq 0 \mid g(t) = y \right\}.
\end{equation*}}
\end{obs}

\begin{proof}
Let $i_0 \in \left\{1, 2\right\}$ be the index satisfying $\newest{\E \left\Vert x_0 - z \right\Vert^p} \in \cA_{i_0}$. \newest{Let $Y = (Y_t)_{t \geq 0}$ be the unique strong solution to 
\begin{equation*}
Y_t = x_0 + \int_0^t \sigma_{i_0}(s, Y_s, \P_{Y_s}) \od B_s + \int_0^t b(s, Y_s, \P_{Y_s}) \od s, \quad t \geq 0,
\end{equation*}
which exists by Theorem \ref{th_strong_existence}, so that  
\begin{equation} \label{finite_moment_again}
\E \sup_{t \in [0, T]} \left\Vert Y_t - z \right\Vert^p < \infty
\end{equation}
for all $T \in (0, \infty)$. Let $T \in (0, \infty]$ such that $\E \left\Vert Y_t - z \right\Vert^p \in \cA_{i_0}$ for all $t \in [0, T)$. Since the function $t \mapsto E \left\Vert Y_t - z \right\Vert^p$ is continuous, we can let $X_t := Y_t$ for $t \in [0, T)$ to obtain a strongly unique $L_p$-solution $(T, (X_t)_{t \in [0, T)})$ to \eqref{sde_general_single_level}, because now $g(t) \in \cA_{i_0}$ for all $t \in [0, T)$.} If $T_0 = \infty$, then clearly we can take $T = T_0 = \infty$. If $T_0 < \infty$, then the limit \newest{
\begin{equation*}
\lim_{t \uparrow T_0} g(t) = \lim_{t \uparrow T_0} \E \left\Vert Y_t - z \right\Vert^p = \E \left\Vert Y_{T_0} - z \right\Vert^p < \infty
\end{equation*}}
exists \newest{by \eqref{finite_moment_again}}, so we can extend the solution to the closed interval $[0, T_0]$.
\end{proof}

Observation \ref{obs_some_solution} implies that if one wants to find conditions for a finite maximal lifetime for the equation \eqref{sde_general_single_level}, then it is sufficient to study the behavior of the moment function $g$ near the level $y$ and show that the solution cannot be extended to $[0, T_0 + \delta)$ for any $\delta > 0$. \Newest{In other words, one needs to show that the equation 
\begin{equation*}
\begin{cases}
\begin{aligned}
& X_t = X_{T_0} + \sum_{i =1 }^2 \int_{T_0}^t \ind_{\left\lbrace g(s) \in \cA_i \right\rbrace } \sigma_i(s, X_s, \P_{X_s}) \od B_s + \int_{T_0}^t b(s, X_s, \P_{X_s}) \od s, \quad t \geq T_0, \\
& g(t) = \E \left\Vert X_t - z \right\Vert^p
\end{aligned}
\end{cases}
\end{equation*}
does not have an $L_p$-solution. In this paper we only consider the case when $p=2$ and the coefficient functions $\sigma_1, \sigma_2$ and $b$ depend only on the time variable.
}

\subsection{Coefficients depending only on the time variable} \label{subsec_nonexistence_time}

Let $p = 2$. Assume that the functions $\sigma_i : [0, \infty) \to \R^{d \times d}$, $i = 1, 2$ and $b : [0, \infty) \to \R^d$ are bounded and Borel measurable. We study the existence of an $L_2$-solution to the equation 
\begin{equation} \label{sde_base_2}
\begin{cases}
\begin{aligned}
& X_t = \revised{x_0} + \sum_{i=1}^2 \int_{T_0}^t  \ind_{ \left\lbrace g(s) \in \cA_i \right\rbrace } \sigma_i(s) \od B_s + \int_{T_0}^t b(s) \od s, \quad t \geq T_0, \\
& g(t) = \E \left\Vert X_t - z\right\Vert^2,
\end{aligned}
\end{cases}
\end{equation}
for some $T_0 \geq 0$, where the initial value $\revised{x_0} \in L_2(\Omega, \cF_{T_0}, \P)$ satisfies $\E \left\Vert \revised{x_0} - z \right\Vert^2 = y$. We notice that if the equation \eqref{sde_base_2} has an $L_2$-solution $(T, (X_t)_{t \in [T_0,  T_1)})$ for some $T_1 > T_0$, then the function $g$ solves the integral equation 
\begin{equation} \label{ode_base}
g(t) = \E \left\Vert \revised{x_0} - z + \int_{T_0}^t b(s) \od s \right\Vert^2 + \sum_{i=1}^2 \int_{T_0}^t \ind_{ \left\lbrace g(s) \in \cA_i \right\rbrace  } \left\Vert \sigma_i(s) \right\Vert^2 \od s, \quad t \in [T_0, T_1).
\end{equation}
Particularly one sees that if \eqref{ode_base} does not have a solution, then the SDE \eqref{sde_base_2} does not have an $L_2$-solution, either.

To be able to formulate sufficient conditions for \newest{the} non-existence of a solution to \eqref{ode_base} we define functions $g_i : [T_0, \infty) \to [0, \infty)$, $i = 1,2$, by
\begin{equation} \label{func_g_k}
g_i(t) := \E \left\Vert \revised{x_0} - z + \int_{T_0}^t b(s) \od s \right\Vert^2 + \int_{T_0}^t \left\Vert \sigma_i(s) \right\Vert^2 \od s
\end{equation}
for $i = 1, 2$. We notice that these functions are defined for all $t \in [T_0, \infty)$ even if the equation \eqref{ode_base} does not have a solution. To motivate the definition of these functions we make the following observation:

\begin{lemm} \label{lemma_time_difference}
Assume that the equation \eqref{ode_base} has a solution on $[T_0, T_1)$ for some $T_1 > T_0$. Suppose that there are a $k \in \left\{1, 2\right\}$ and $r, s \in [T_0, T_1)$ such that $r < s$ and $g(t) \in \cA_k$ for all $t \in (r, s]$. Then 
\begin{equation*}
g(t) - g(r) = g_k(t) - g_k(r)
\end{equation*}
for all $t \in (r, s]$.
\end{lemm}

\begin{proof}
Let $A(t) := \E \left\Vert \revised{x_0} - z + \int_{T_0}^t b(s) \od s \right\Vert^2$. Then 
\begin{align*}
g(t) - g(r) & = A(t) - A(r) + \sum_{i=1}^2 \int_r^t \ind_{ \left\lbrace g(s) \in \cA_i \right\rbrace  } \left\Vert \sigma_i(s) \right\Vert^2 \od s \\
& = A(t) - A(r) + \int_r^t \left\Vert \sigma_k(s) \right\Vert^2 \od s \\
& = g_k(t) - g_k(r).
\end{align*}
\end{proof}

The proposition below describes cases when the integral equation \eqref{ode_base} does not have a solution.
\begin{prop} \label{th_time_no_sol}
Assume that there exists an $\varepsilon > 0$ such that
\begin{enumerate}[label=\rm{(\Alph*)}]
    \item if $\cA_1 = [0, y)$, then \label{nosol_as_A}
    \begin{enumerate}[label=\small{\rm{(A\arabic*)}}]
        \item the function $g_1$ is non-decreasing on $[T_0, T_0 + \varepsilon]$, and \label{cond_A_1}
        \item the function $g_2$ is strictly decreasing on $[T_0, T_0 + \varepsilon]$,  \label{cond_A_2}
    \end{enumerate}
    \item if $\cA_1 = [0, y]$, then \label{nosol_as_B}
    \begin{enumerate}[label=\small{\rm{(B\arabic*)}}] 
        \item the function $g_1$ is strictly increasing on $[T_0, T_0 + \varepsilon]$, and  \label{cond_B_1}
        \item the function $g_2$ is non-increasing on $[T_0, T_0 + \varepsilon]$.  \label{cond_B_2}
    \end{enumerate}

\end{enumerate}
Then the integral equation \eqref{ode_base} does not have a solution.

\end{prop}

\begin{proof}
\newest{Let us assume that \eqref{ode_base} has a solution $g$ on the interval $[T_0, T_0 + \delta)$ for some $\delta > 0$, and fix any $t_1 \in (T_0, T_0 + \delta)$.}

First we assume that $g(t_1) \neq y$. Then, by the continuity of the function $g$, there exists a $t_0 \in [\newest{T_0}, t_1)$ such that $g(t_0) = y$ and $g(t) \neq y$ for all $t \in (t_0, t_1]$. Now we have two possible cases:
\begin{enumerate} [label=(\roman*)]
    \item $g(t_1) < y$, and \label{nosol1_proof_case1}
    \item $g(t_1) > y$. \label{nosol1_proof_case2}
\end{enumerate}
In the first case we notice that for any $t \in (t_0, t_1)$ one has
\begin{align*}
0 > g(t) - y = g(t) - g(t_0) = g_1(t) - g_1(t_0) 
\end{align*} 
by Lemma \ref{lemma_time_difference}. However, by assumptions \ref{cond_A_1} and \ref{cond_B_1} the function $g_1$ is either non-decreasing or strictly increasing, hence $g_1(t) - g_1(t_0) \geq 0$, which is a contradiction.

In the case \ref{nosol1_proof_case2} we use a similar argument: for any $t \in (t_0, t_1)$ we have 
\begin{equation*}
0 < g(t) - y = g(t) - g(t_0) = g_2(t) - g_2(t_0), 
\end{equation*}
by Lemma \ref{lemma_time_difference}, but $g_2$ is either \newest{strictly decreasing or non-increasing} by assumptions \ref{cond_A_2} and \ref{cond_B_2}, hence $g_2(t) - g_2(t_0) \leq 0$, which is a contradiction.

The remaining case is that $g(t) = y$ for all $t \in [T_0, t_1]$. However, now
\begin{equation*}
0 = g(t) - y = g(t) - g(t_0) = g_i(t) - g_i(t_0),
\end{equation*}
where $i \in \left\{1, 2\right\}$ satisfies $y \in \cA_i$. The function $g_i$ is strictly monotone for both $i=1$ and $i=2$ by assumptions \ref{cond_A_2} and \ref{cond_B_1}, hence $g_i(t) - g_i(t) \neq 0$, which is a contradiction.

We conclude that the integral equation \eqref{ode_base} has no solution.
\end{proof}

As a corollary we obtain the non-existence of an $L_2$-solution for the SDE \eqref{sde_base_2} under the same assumptions.

\begin{theo} \label{coro_no_solo}
Suppose that the assumptions in Proposition \ref{th_time_no_sol} hold. Then the SDE \eqref{sde_base_2} does not have an $L_2$-solution.
\end{theo}

\begin{proof}
\newest{This} follows immediately from Proposition \ref{th_time_no_sol} by noticing that if the SDE \eqref{sde_base_2} has an $L_2$-solution $(T_1, (X_t)_{t \in [T_0, T_1)} )$, then also the integral equation \eqref{sde_base_2} needs to have a solution on $[T_0, T_1)$. 
\end{proof}

\newest{
The next example demonstrates how the existence of a solution to \eqref{sde_base_2} can depend on the initial value.
\begin{exam}
Consider the one-dimensional SDE 
\begin{equation} \label{sde_simple_ex}
\begin{cases}
\begin{aligned}
& X_t = 1 + \sqrt{2} \int_{T_0}^t  \ind_{ \left\lbrace g(s) < 1 \right\rbrace } \od B_s - \int_{T_0}^t 1 \od s, \quad t \geq T_0 \geq 0, \\
& g(t) = \E \left| X_t \right|^2.
\end{aligned}
\end{cases}
\end{equation}
Here $y = 1$, $\revised{x_0} \equiv 1$, $z=0$, $\sigma_1 \equiv \sqrt{2}$, $\sigma_2 \equiv 0$ and $b \equiv -1$. One computes that $g_1(t) = 1 + (t-T_0)^2$, which is non-decreasing everywhere, and $g_2(t) = [1 - (t-T_0)]^2$, which is strictly decreasing on $[T_0, T_0 + 1]$, so by Theorem \ref{coro_no_solo} the SDE \eqref{sde_simple_ex} does not have an $L_2$-solution. 

The statement above holds for arbitrary $T_0 \geq 0$. Next, let us consider the following SDE:
\begin{equation} \label{sde_simple_ex2}
\tilde{X}_t = \sqrt{2} \int_{0}^t  \ind_{ \left\lbrace \E \left| \tilde{X}_s \right|^2 < 1 \right\rbrace } \od B_s - \int_{0}^t 1 \od s, \quad t \geq 0.
\end{equation}
We observe that the process $\tilde{X}_t = \sqrt{2} B_{t \wedge S_0} - t$, $t \geq 0$, solves \eqref{sde_simple_ex2}, where 
\begin{equation*}
S_0 := \inf \left\lbrace t \geq 0 \hspace{0.25em} \Big| \hspace{0.25em} \E | \tilde{X}_t |^2 = 1 \right\rbrace = \sqrt{2} - 1.
\end{equation*} 
Moreover, on $[S_0, \infty)$ the process $\tilde{X}$ solves
\begin{equation} \label{sde_simple_ex2_T0}
\begin{cases}
\begin{aligned}
& \tilde{X}_t = \tilde{X}_{S_0} + \sqrt{2} \int_{S_0}^t \ind_{ \left\lbrace \tilde{g}(s) < 1 \right\rbrace } \od B_s - \int_{S_0}^t 1 \od s, \quad t \geq S_0, \\
& \tilde{g}(t) = \E \left| \tilde{X}_t \right|^2,
\end{aligned}
\end{cases}
\end{equation}
which corresponds with \eqref{sde_simple_ex}, but with a different initial value. Clearly it holds $\E \left| \tilde{X}_{S_0} \right|^2 = 1$, but because $\E \tilde{X}_{S_0} = 1 - \sqrt{2} = - S_0$, the function
\begin{equation*}
\tilde{g}_2(t) = \E \left| \tilde{X}_{S_0} - (t - S_0) \right|^2 =(1 - S_0^2) + t^2
\end{equation*}
is strictly increasing for $t \geq S_0$. This shows that changing the initial value in \eqref{sde_simple_ex} affects the existence of a solution, even if the $L_2$-norm \Newest{of the initial condition} remains the same, \Newest{because the function $g_2$ is strictly increasing on $[S_0, S_0 + 1]$, but $\tilde{g}_2$ is non-decreasing everywhere.}


\end{exam}
}

\subsection{Oscillating coefficients}

In Theorem \ref{coro_no_solo} we require monotonicity of the functions $g_1$ and $g_2$ on some interval $[T_0, T_0 + \varepsilon)$ for some $\varepsilon > 0$. Next we construct an example showing that the monotonicity of the function $g_1$ is not necessary to obtain the conclusion of Theorem \ref{coro_no_solo} if the function $g_1$ is oscillating in a sufficient way.

Let $a_n := \frac{1}{2^n}$ for $n \in \N \cup \left\{ 0 \right\}$ and define $\Delta_n := \frac{1}{4} \left( a_{n-1} - a_n \right) = \frac{1}{2^{n+2}}$ for $n \in \N$. Let us consider the equation 
\begin{equation} \label{sde_oscillation_base}
\begin{cases}
\begin{aligned}
& X_t = 1 + \int_0^t \ind_{ \left\lbrace g(s) < 1 \right\rbrace }  \sqrt{2} \left(   \sum_{n=0}^\infty  \ind_{ \left\{ s \in (a_n + \Delta_n, a_n + 3 \Delta_n] \right\} }  \right) \od B_s - \int_0^t \frac{ \ind_{\left\{ s < 1 - \alpha \right\}} }{2 \sqrt{1 - s}} \od s, \\
& g(t) = \E \left| X_t \right|^2,
\end{aligned}
\end{cases}
\end{equation}
where the constant $\alpha \in (0, 1)$ is \newest{chosen} to make the drift coefficient
\begin{equation*}
b(t) = \frac{ \ind_{\left\{ t < 1 - \alpha \right\}} }{2 \sqrt{1 - t}}
\end{equation*}
bounded. The diffusion coefficients are now
\begin{equation*}
\sigma_1(t) = \sqrt{2} \left(  \sum_{n=0}^\infty  \ind_{ \left\{ t \in (a_n + \Delta_n, a_n + 3 \Delta_n] \right\} }  \right) 
\end{equation*}
and $\sigma_2 \equiv 0$. The function $\sigma_1$ is illustrated in Figure \eqref{osc_figure1} above. We see that equation \eqref{sde_oscillation_base} corresponds with equation \eqref{sde_base_2} by letting $T_0 = 0$, $y = 1$, $z = 0$ and $\revised{x_0} \equiv 1$. \newest{Moreover, we have $\cA_1 = [0, 1)$ and $\cA_2 = [1, \infty)$.}

If the SDE \eqref{sde_oscillation_base} has an $L_2$-solution $(T, (X_t)_{t \in [0, T)})$, where $T \in (0, 1 - \alpha)$, then the moment function $g$ solves the integral equation 
\begin{equation} \label{ode_oscillation}
g(t) = 1 - t + 2 \sum_{n=1}^\infty \int_0^t \ind_{ \left\lbrace g(s) < 1 \right\rbrace } \ind_{ \left\{ s \in (a_n + \Delta_n, a_n + 3 \Delta_n] \right\} } \od s, \quad t \in [0, T).
\end{equation}
\newest{We get
\begin{equation*}
g_1(t) = 1 - \left( t \wedge 1 \right) + 2 \sum_{n=1}^\infty \int_0^t \ind_{ \left\{ s \in (a_n + \Delta_n, a_n + 3 \Delta_n] \right\} } \od s
\end{equation*}
and 
\begin{equation*}
g_2(t) = 1 - (t \wedge 1)
\end{equation*}
for $t \in [0, T)$.} We see that the condition \ref{cond_A_1} in Proposition \ref{th_time_no_sol} is not satisfied in $T_0 = 0$ since the function $g_1$ is not monotone on $[0, \varepsilon]$ for any $\varepsilon > 0$ as seen in Figure \eqref{osc_figure2} above. However, the equation \eqref{sde_oscillation_base} still does not have a solution.

\renewcommand{\thesubfigure}{\alph{subfigure}}
%
%
\begin{figure}%
\centering

\subfloat[][\centering  \label{osc_figure1}]{
\begin{tikzpicture}[scale = 0.75]
\begin{axis}[
   xmin = 0, xmax = 1,
   ymin = -0.25, ymax = 2] 
	\draw[line width=0.25mm, black , dotted] (axis cs:0,0) -- node[left]{} (axis cs:1,0);
	\draw[line width=0.5mm, blue ] (axis cs:0.625,1.4142135623730951) -- node[left]{} (axis cs:0.875,1.4142135623730951);
	\draw[line width=0.5mm, blue ] (axis cs:0.875,0) -- node[left]{} (axis cs:1,0);
	\draw[line width=0.5mm, blue ] (axis cs:0.5,0) -- node[left]{} (axis cs:0.625,0);
	\draw[line width=0.5mm, blue ] (axis cs:0.3125,1.4142135623730951) -- node[left]{} (axis cs:0.4375,1.4142135623730951);
	\draw[line width=0.5mm, blue ] (axis cs:0.4375,0) -- node[left]{} (axis cs:0.5,0);
	\draw[line width=0.5mm, blue ] (axis cs:0.25,0) -- node[left]{} (axis cs:0.3125,0);
	\draw[line width=0.5mm, blue ] (axis cs:0.15625,1.4142135623730951) -- node[left]{} (axis cs:0.21875,1.4142135623730951);
	\draw[line width=0.5mm, blue ] (axis cs:0.21875,0) -- node[left]{} (axis cs:0.25,0);
	\draw[line width=0.5mm, blue ] (axis cs:0.125,0) -- node[left]{} (axis cs:0.15625,0);
	\draw[line width=0.5mm, blue ] (axis cs:0.078125,1.4142135623730951) -- node[left]{} (axis cs:0.109375,1.4142135623730951);
	\draw[line width=0.5mm, blue ] (axis cs:0.109375,0) -- node[left]{} (axis cs:0.125,0);
	\draw[line width=0.5mm, blue ] (axis cs:0.0625,0) -- node[left]{} (axis cs:0.078125,0);
	\draw[line width=0.5mm, blue ] (axis cs:0.0390625,1.4142135623730951) -- node[left]{} (axis cs:0.0546875,1.4142135623730951);
	\draw[line width=0.5mm, blue ] (axis cs:0.0546875,0) -- node[left]{} (axis cs:0.0625,0);
	\draw[line width=0.5mm, blue ] (axis cs:0.03125,0) -- node[left]{} (axis cs:0.0390625,0);
	\draw[line width=0.5mm, blue ] (axis cs:0.01953125,1.4142135623730951) -- node[left]{} (axis cs:0.02734375,1.4142135623730951);
	\draw[line width=0.5mm, blue ] (axis cs:0.02734375,0) -- node[left]{} (axis cs:0.03125,0);
	\draw[line width=0.5mm, blue ] (axis cs:0.015625,0) -- node[left]{} (axis cs:0.01953125,0);
	\draw[line width=0.5mm, blue ] (axis cs:0.009765625,1.4142135623730951) -- node[left]{} (axis cs:0.013671875,1.4142135623730951);
	\draw[line width=0.5mm, blue ] (axis cs:0.013671875,0) -- node[left]{} (axis cs:0.015625,0);
	\draw[line width=0.5mm, blue ] (axis cs:0.0078125,0) -- node[left]{} (axis cs:0.009765625,0);
	\draw[line width=0.5mm, blue ] (axis cs:0.0048828125,1.4142135623730951) -- node[left]{} (axis cs:0.0068359375,1.4142135623730951);
	\draw[line width=0.5mm, blue ] (axis cs:0.0068359375,0) -- node[left]{} (axis cs:0.0078125,0);
	\draw[line width=0.5mm, blue ] (axis cs:0.00390625,0) -- node[left]{} (axis cs:0.0048828125,0);
	\draw[line width=0.5mm, blue ] (axis cs:0.00244140625,1.4142135623730951) -- node[left]{} (axis cs:0.00341796875,1.4142135623730951);
	\draw[line width=0.5mm, blue ] (axis cs:0.00341796875,0) -- node[left]{} (axis cs:0.00390625,0);
	\draw[line width=0.5mm, blue ] (axis cs:0.001953125,0) -- node[left]{} (axis cs:0.00244140625,0);
	\draw[line width=0.5mm, blue ] (axis cs:0.001220703125,1.4142135623730951) -- node[left]{} (axis cs:0.001708984375,1.4142135623730951);
	\draw[line width=0.5mm, blue ] (axis cs:0.001708984375,0) -- node[left]{} (axis cs:0.001953125,0);
	\draw[line width=0.5mm, blue ] (axis cs:0.0009765625,0) -- node[left]{} (axis cs:0.001220703125,0);
	\draw[line width=0.5mm, blue ] (axis cs:0.0006103515625,1.4142135623730951) -- node[left]{} (axis cs:0.0008544921875,1.4142135623730951);
	\draw[line width=0.5mm, blue ] (axis cs:0.0008544921875,0) -- node[left]{} (axis cs:0.0009765625,0);
	\draw[line width=0.5mm, blue ] (axis cs:0.00048828125,0) -- node[left]{} (axis cs:0.0006103515625,0);
	\draw[line width=0.5mm, blue ] (axis cs:0.00030517578125,1.4142135623730951) -- node[left]{} (axis cs:0.00042724609375,1.4142135623730951);
	\draw[line width=0.5mm, blue ] (axis cs:0.00042724609375,0) -- node[left]{} (axis cs:0.00048828125,0);
	\draw[line width=0.5mm, blue ] (axis cs:0.000244140625,0) -- node[left]{} (axis cs:0.00030517578125,0);
	\draw[line width=0.5mm, blue ] (axis cs:0.000152587890625,1.4142135623730951) -- node[left]{} (axis cs:0.000213623046875,1.4142135623730951);
	\draw[line width=0.5mm, blue ] (axis cs:0.000213623046875,0) -- node[left]{} (axis cs:0.000244140625,0);
	\draw[line width=0.5mm, blue ] (axis cs:0.0001220703125,0) -- node[left]{} (axis cs:0.000152587890625,0);
	\draw[line width=0.5mm, blue ] (axis cs:0.0000762939453125,1.4142135623730951) -- node[left]{} (axis cs:0.0001068115234375,1.4142135623730951);
	\draw[line width=0.5mm, blue ] (axis cs:0.0001068115234375,0) -- node[left]{} (axis cs:0.0001220703125,0);
	\draw[line width=0.5mm, blue ] (axis cs:0.00006103515625,0) -- node[left]{} (axis cs:0.0000762939453125,0);
	\draw[line width=0.5mm, blue ] (axis cs:0.00003814697265625,1.4142135623730951) -- node[left]{} (axis cs:0.00005340576171875,1.4142135623730951);
	\draw[line width=0.5mm, blue ] (axis cs:0.00005340576171875,0) -- node[left]{} (axis cs:0.00006103515625,0);
	\draw[line width=0.5mm, blue ] (axis cs:0.000030517578125,0) -- node[left]{} (axis cs:0.00003814697265625,0);
	\draw[line width=0.5mm, blue ] (axis cs:0.000019073486328125,1.4142135623730951) -- node[left]{} (axis cs:0.000026702880859375,1.4142135623730951);
	\draw[line width=0.5mm, blue ] (axis cs:0.000026702880859375,0) -- node[left]{} (axis cs:0.000030517578125,0);
	\draw[line width=0.5mm, blue ] (axis cs:0.0000152587890625,0) -- node[left]{} (axis cs:0.000019073486328125,0);
\end{axis}
\end{tikzpicture}
}
\quad
\subfloat[][\centering  \label{osc_figure2}]{
\begin{tikzpicture}[scale = 0.75]
\begin{axis}[
   xmin = 0, xmax = 1,
   ymin = 0.5, ymax = 1.5] 
	\draw[line width=0.25mm, black , dotted] (axis cs:0,1) -- node[left]{} (axis cs:1,1);
	\draw[line width=0.5mm, blue , cap=round] (axis cs:1,1) -- node[left]{} (axis cs:0.875,1.125);
	\draw[line width=0.5mm, blue , cap=round] (axis cs:0.875,1.125) -- node[left]{} (axis cs:0.75,1);
	\draw[line width=0.5mm, blue , cap=round] (axis cs:0.75,1) -- node[left]{} (axis cs:0.625,0.875);
	\draw[line width=0.5mm, blue , cap=round] (axis cs:0.625,0.875) -- node[left]{} (axis cs:0.5,1);
	\draw[line width=0.5mm, blue , cap=round] (axis cs:0.5,1) -- node[left]{} (axis cs:0.4375,1.0625);
	\draw[line width=0.5mm, blue , cap=round] (axis cs:0.4375,1.0625) -- node[left]{} (axis cs:0.375,1);
	\draw[line width=0.5mm, blue , cap=round] (axis cs:0.375,1) -- node[left]{} (axis cs:0.3125,0.9375);
	\draw[line width=0.5mm, blue , cap=round] (axis cs:0.3125,0.9375) -- node[left]{} (axis cs:0.25,1);
	\draw[line width=0.5mm, blue , cap=round] (axis cs:0.25,1) -- node[left]{} (axis cs:0.21875,1.03125);
	\draw[line width=0.5mm, blue , cap=round] (axis cs:0.21875,1.03125) -- node[left]{} (axis cs:0.1875,1);
	\draw[line width=0.5mm, blue , cap=round] (axis cs:0.1875,1) -- node[left]{} (axis cs:0.15625,0.96875);
	\draw[line width=0.5mm, blue , cap=round] (axis cs:0.15625,0.96875) -- node[left]{} (axis cs:0.125,1);
	\draw[line width=0.5mm, blue , cap=round] (axis cs:0.125,1) -- node[left]{} (axis cs:0.109375,1.015625);
	\draw[line width=0.5mm, blue , cap=round] (axis cs:0.109375,1.015625) -- node[left]{} (axis cs:0.09375,1);
	\draw[line width=0.5mm, blue , cap=round] (axis cs:0.09375,1) -- node[left]{} (axis cs:0.078125,0.984375);
	\draw[line width=0.5mm, blue , cap=round] (axis cs:0.078125,0.984375) -- node[left]{} (axis cs:0.0625,1);
	\draw[line width=0.5mm, blue , cap=round] (axis cs:0.0625,1) -- node[left]{} (axis cs:0.0546875,1.0078125);
	\draw[line width=0.5mm, blue , cap=round] (axis cs:0.0546875,1.0078125) -- node[left]{} (axis cs:0.046875,1);
	\draw[line width=0.5mm, blue , cap=round] (axis cs:0.046875,1) -- node[left]{} (axis cs:0.0390625,0.9921875);
	\draw[line width=0.5mm, blue , cap=round] (axis cs:0.0390625,0.9921875) -- node[left]{} (axis cs:0.03125,1);
	\draw[line width=0.5mm, blue , cap=round] (axis cs:0.03125,1) -- node[left]{} (axis cs:0.02734375,1.00390625);
	\draw[line width=0.5mm, blue , cap=round] (axis cs:0.02734375,1.00390625) -- node[left]{} (axis cs:0.0234375,1);
	\draw[line width=0.5mm, blue , cap=round] (axis cs:0.0234375,1) -- node[left]{} (axis cs:0.01953125,0.99609375);
	\draw[line width=0.5mm, blue , cap=round] (axis cs:0.01953125,0.99609375) -- node[left]{} (axis cs:0.015625,1);
	\draw[line width=0.5mm, blue , cap=round] (axis cs:0.015625,1) -- node[left]{} (axis cs:0.013671875,1.001953125);
	\draw[line width=0.5mm, blue , cap=round] (axis cs:0.013671875,1.001953125) -- node[left]{} (axis cs:0.01171875,1);
	\draw[line width=0.5mm, blue , cap=round] (axis cs:0.01171875,1) -- node[left]{} (axis cs:0.009765625,0.998046875);
	\draw[line width=0.5mm, blue , cap=round] (axis cs:0.009765625,0.998046875) -- node[left]{} (axis cs:0.0078125,1);
	\draw[line width=0.5mm, blue , cap=round] (axis cs:0.0078125,1) -- node[left]{} (axis cs:0.0068359375,1.0009765625);
	\draw[line width=0.5mm, blue , cap=round] (axis cs:0.0068359375,1.0009765625) -- node[left]{} (axis cs:0.005859375,1);
	\draw[line width=0.5mm, blue , cap=round] (axis cs:0.005859375,1) -- node[left]{} (axis cs:0.0048828125,0.9990234375);
	\draw[line width=0.5mm, blue , cap=round] (axis cs:0.0048828125,0.9990234375) -- node[left]{} (axis cs:0.00390625,1);
	\draw[line width=0.5mm, blue , cap=round] (axis cs:0.00390625,1) -- node[left]{} (axis cs:0.00341796875,1.00048828125);
	\draw[line width=0.5mm, blue , cap=round] (axis cs:0.00341796875,1.00048828125) -- node[left]{} (axis cs:0.0029296875,1);
	\draw[line width=0.5mm, blue , cap=round] (axis cs:0.0029296875,1) -- node[left]{} (axis cs:0.00244140625,0.99951171875);
	\draw[line width=0.5mm, blue , cap=round] (axis cs:0.00244140625,0.99951171875) -- node[left]{} (axis cs:0.001953125,1);
	\draw[line width=0.5mm, blue , cap=round] (axis cs:0.001953125,1) -- node[left]{} (axis cs:0.001708984375,1.000244140625);
	\draw[line width=0.5mm, blue , cap=round] (axis cs:0.001708984375,1.000244140625) -- node[left]{} (axis cs:0.00146484375,1);
	\draw[line width=0.5mm, blue , cap=round] (axis cs:0.00146484375,1) -- node[left]{} (axis cs:0.001220703125,0.999755859375);
	\draw[line width=0.5mm, blue , cap=round] (axis cs:0.001220703125,0.999755859375) -- node[left]{} (axis cs:0.0009765625,1);
	\draw[line width=0.5mm, blue , cap=round] (axis cs:0.0009765625,1) -- node[left]{} (axis cs:0.0008544921875,1.0001220703125);
	\draw[line width=0.5mm, blue , cap=round] (axis cs:0.0008544921875,1.0001220703125) -- node[left]{} (axis cs:0.000732421875,1);
	\draw[line width=0.5mm, blue , cap=round] (axis cs:0.000732421875,1) -- node[left]{} (axis cs:0.0006103515625,0.9998779296875);
	\draw[line width=0.5mm, blue , cap=round] (axis cs:0.0006103515625,0.9998779296875) -- node[left]{} (axis cs:0.00048828125,1);
	\draw[line width=0.5mm, blue , cap=round] (axis cs:0.00048828125,1) -- node[left]{} (axis cs:0.00042724609375,1.00006103515625);
	\draw[line width=0.5mm, blue , cap=round] (axis cs:0.00042724609375,1.00006103515625) -- node[left]{} (axis cs:0.0003662109375,1);
	\draw[line width=0.5mm, blue , cap=round] (axis cs:0.0003662109375,1) -- node[left]{} (axis cs:0.00030517578125,0.99993896484375);
	\draw[line width=0.5mm, blue , cap=round] (axis cs:0.00030517578125,0.99993896484375) -- node[left]{} (axis cs:0.000244140625,1);
	\draw[line width=0.5mm, blue , cap=round] (axis cs:0.000244140625,1) -- node[left]{} (axis cs:0.000213623046875,1.000030517578125);
	\draw[line width=0.5mm, blue , cap=round] (axis cs:0.000213623046875,1.000030517578125) -- node[left]{} (axis cs:0.00018310546875,1);
	\draw[line width=0.5mm, blue , cap=round] (axis cs:0.00018310546875,1) -- node[left]{} (axis cs:0.000152587890625,0.999969482421875);
	\draw[line width=0.5mm, blue , cap=round] (axis cs:0.000152587890625,0.999969482421875) -- node[left]{} (axis cs:0.0001220703125,1);
	\draw[line width=0.5mm, blue , cap=round] (axis cs:0.0001220703125,1) -- node[left]{} (axis cs:0.0001068115234375,1.0000152587890625);
	\draw[line width=0.5mm, blue , cap=round] (axis cs:0.0001068115234375,1.0000152587890625) -- node[left]{} (axis cs:0.000091552734375,1);
	\draw[line width=0.5mm, blue , cap=round] (axis cs:0.000091552734375,1) -- node[left]{} (axis cs:0.0000762939453125,0.9999847412109375);
	\draw[line width=0.5mm, blue , cap=round] (axis cs:0.0000762939453125,0.9999847412109375) -- node[left]{} (axis cs:0.00006103515625,1);
	\draw[line width=0.5mm, blue , cap=round] (axis cs:0.00006103515625,1) -- node[left]{} (axis cs:0.00005340576171875,1.0000076293945312);
	\draw[line width=0.5mm, blue , cap=round] (axis cs:0.00005340576171875,1.0000076293945312) -- node[left]{} (axis cs:0.0000457763671875,1);
	\draw[line width=0.5mm, blue , cap=round] (axis cs:0.0000457763671875,1) -- node[left]{} (axis cs:0.00003814697265625,0.9999923706054688);
	\draw[line width=0.5mm, blue , cap=round] (axis cs:0.00003814697265625,0.9999923706054688) -- node[left]{} (axis cs:0.000030517578125,1);
	\draw[line width=0.5mm, blue , cap=round] (axis cs:0.000030517578125,1) -- node[left]{} (axis cs:0.000026702880859375,1.0000038146972656);
	\draw[line width=0.5mm, blue , cap=round] (axis cs:0.000026702880859375,1.0000038146972656) -- node[left]{} (axis cs:0.00002288818359375,1);
	\draw[line width=0.5mm, blue , cap=round] (axis cs:0.00002288818359375,1) -- node[left]{} (axis cs:0.000019073486328125,0.9999961853027344);
	\draw[line width=0.5mm, blue , cap=round] (axis cs:0.000019073486328125,0.9999961853027344) -- node[left]{} (axis cs:0.0000152587890625,1);
\end{axis}
\end{tikzpicture}
}

\caption{}
\end{figure}

\begin{theo} \label{th_osc_no_sol}
The equation \eqref{sde_oscillation_base} does not have an $L_2$-solution.
\end{theo}

We first prove the following lemma.

\begin{lemm} \label{lemma_delayed_osc}
For $s \in (0, 1)$ consider the equation 
\begin{equation} \label{ode_delayed_oscillation}
g_s(t) = 1 - \left(t - s\right) + 2 \sum_{n=1}^\infty \int_s^t \ind_{ \left\lbrace g_s(u) < 1 \right\rbrace } \ind_{ \left\{ u \in (a_n + \Delta_n, a_n + 3 \Delta_n] \right\} } \od u, \quad t \in [s, 1).
\end{equation}
Let 
\begin{equation*}
m(s) := \sup \left\{ t \in [s, 1) \mid \textrm{ \eqref{ode_delayed_oscillation} has a solution on } [s, t] \right\}
\end{equation*}
and $\kappa(s) := \inf \left\{ k \geq 0 \mid s \geq a_k \right\} $. Then, if $\kappa(s) > 1$, one has 
\begin{equation*}
m(s) \leq a_{\kappa(s)} + 9 \Delta_{\kappa(s)}.
\end{equation*}
In particular, $m(s) \to 0$ as $s \downarrow 0$.

\end{lemm}

\begin{proof}
We start by defining the functions
\begin{equation*}
g_{s, 1} (t) := 1 - \left(  t - s \right) + 2 \sum_{n=1}^\infty \int_s^t \ind_{ \left\{ u \in (a_n + \Delta_n, a_n + 3 \Delta_n] \right\} } \od u
\end{equation*}
and 
\begin{equation*}
g_{s, 2}(t) := 1 - \left( t - s \right)
\end{equation*}
for $t \geq s$. These functions correspond with the functions $g_1$ and $g_2$ defined in \eqref{func_g_k}. Let \newest{$s \leq u < t < 1$} and suppose that the equation \eqref{ode_delayed_oscillation} has a solution on $[s, t]$. Then, similarly as in Lemma \ref{lemma_time_difference}, we observe the following:
\begin{enumerate}[label=(O\arabic*)]
    \item If $u, t \in [a_n - \Delta_{n+1}, a_n + \Delta_n]$ for any $n \in \N$, then 
    \begin{equation*}
    g_s(t) - g_s(u) = g_{s, 2}(t) - g_{s,2}(u) = -(t - u).   
    \end{equation*} \label{osc_obs_1}
    \item If $u, t \in [a_n + \Delta_n, a_n + 3 \Delta_n]$ for any $n \in \N$, then 
    \begin{equation*}
    g_s(t) - g_s(u) = g_{s, 1}(t) - g_{s,1}(u) = t - u.   
    \end{equation*} \label{osc_obs_2}
\end{enumerate}
\revised{These properties follow from that fact that on the intervals $(a_n - \Delta_{n+1}, a_n + \Delta_n]$ the diffusion term is switched off, so $g_s$ is strictly decreasing, but on the intervals $(a_n + \Delta_n, a_n + 3 \Delta_n]$ the diffusion term is switched on, so $g$ is strictly increasing.}

Let $k := \kappa(s)$. We recall that $k > 1$ by assumption. We have the following three possible cases:
\begin{enumerate} [label=(\roman*)]
    \item $s \in [a_k, a_k + \Delta_k)$, \label{osc_lemma_case1}
    \item $s \in [a_k + \Delta_k, a_k + 3 \Delta_k)$, and \label{osc_lemma_case2}
    \item $s \in [a_k + 3 \Delta_k, a_{k-1})$. \label{osc_lemma_case3}
\end{enumerate}
We want to show that in each case one can find an $r \in [s, 1)$ such that $g_s(r) = 1$ and $r \in [a_k + \Delta_k, a_k + 9 \Delta_k)$. We see that the function $g_{s,2}$ is strictly decreasing everywhere and the function $g_{s, 1}$ is non-decreasing only on the set
\begin{equation*}
\left( \bigcup_{n \in \N} [a_n + \Delta_n, a_n + 3 \Delta_n) \right) \cap [s, 1),
\end{equation*}
so we will use Proposition \ref{th_time_no_sol} to show that $m(s) = r$.

\textbf{Case \ref{osc_lemma_case1}:} Since $g_{s, 1}$ is strictly decreasing on $[s, a_k + \Delta_k)$, we know that there exists a solution to \eqref{ode_delayed_oscillation} defined at least up to $a_k + \Delta_k$. By observation \ref{osc_obs_1} \newest{and because \revised{$g_s(s) = 1$}} we notice that
\begin{equation*}
g_s(a_k + \Delta_k) = -(a_k + \Delta_k - s) + g_s(s) = 1 - (a_k + \Delta_k - s) < 1.
\end{equation*}
Let $r := (a_k + \Delta_k) + [(a_k + \Delta_k) - s] = 2(a_k + \Delta_k) - s$. Since 
\begin{equation*}
r \leq 2(a_k + \Delta_k) - a_k = a_k + 2 \Delta_k,
\end{equation*}
it holds that $g_s(t) < 1$ for all $t \in [a_k + \Delta_k, r)$. On the other hand, using observation \ref{osc_obs_2} this time we see that
\begin{align*}
g_s(r) & = r - (a_k + \Delta_k) + g_s(a_k + \Delta_k) \\
& = (a_k + \Delta_k - s) + 1 - (a_k + \Delta_k - s) \\
& = 1,
\end{align*}
so we can apply Proposition \ref{th_time_no_sol} \newest{(A)} to obtain that 
\begin{equation*}
m(s) = r \leq a_k + 2 \Delta_k.
\end{equation*}

\textbf{Case \ref{osc_lemma_case2}:} We can directly apply Proposition \ref{th_time_no_sol} \newest{(A)} to obtain that 
\begin{equation*}
m(s) = s \leq a_k + 3 \Delta_k.
\end{equation*}

\textbf{Case \ref{osc_lemma_case3}:} We use a similar argumentation as in case \ref{osc_lemma_case1}: we have 
\begin{equation*}
g_s(a_{k-1} + \Delta_{k-1}) = -(a_{k-1} + \Delta_{k-1} - s) + g_s(s) = 1 - (a_{k-1} + \Delta_{k-1} - s) < 1.
\end{equation*}
Let $r := (a_{k-1} + \Delta_{k-1}) + (a_{k-1} + \Delta_{k-1}) - s = 2(a_{k-1} + \Delta_{k-1}) - s$. Then
\begin{align*}
g_s(r) & = r - (a_{k-1} + \Delta_k) + g_s(a_{k-1} + \Delta_{k-1}) \\
& = (a_{k-1} + \Delta_{k-1} - s) + 1 - (a_{k-1} + \Delta_{k-1} - s) \\
& = 1.
\end{align*}
Thus by Proposition \ref{th_time_no_sol} \newest{(A)} we obtain that
\begin{align*}
m(s) = r & \leq 2 \left( a_{k-1} + \Delta_{k-1} \right) - a_{k} - 3 \Delta_k \\
& = a_k + 9 \Delta_k. 
\end{align*}

We conclude that 
\begin{equation*}
m(s) \leq a_{\kappa(s)} + 9 \Delta_{\kappa(s)} \to 0
\end{equation*}
as $s \downarrow 0$.

\end{proof}

Now we can proceed to prove Theorem \ref{th_osc_no_sol}.

\begin{proof}[Proof of Theorem \ref{th_osc_no_sol}]

It is sufficient to show that the integral equation \eqref{ode_oscillation} does not have a solution. We \newest{assume} that there does exist a $\delta > 0$ such that the equation \eqref{ode_oscillation} has a solution on $[0, \delta)$.

First, let us \newest{suppose} that 
\begin{enumerate}[label=(\roman*)]
    \item for all $\varepsilon \in (0, \delta)$ the set $\left\{ u \in [0, \varepsilon] \mid g(u) = 1 \right\}$ is infinite.
\end{enumerate}
Then we have a sequence $(s_n)_{n=1}^\infty$ such that $g(s_n) = 1$ for all $n \in \N$ and $s_n \downarrow 0$. However, by Lemma \ref{lemma_delayed_osc} we have $m(s_n) \to 0$ as $n \to \infty$, which implies that there is an $n_0 \in \N$ such that $m(s_{n_0}) < \delta$, which is a contradiction. 

Therefore, there are only finitely many $t \in [0, \delta)$ where $g(t) = 1$. But then there is an $s \in (0, \delta)$ such that either
\begin{enumerate}[label=(\roman*)]
\setcounter{enumi}{1}
    \item $g(t) > 1$ or \label{osc_proof_case_2}
    \item $g(t) < 1$
\end{enumerate}
for all $t \in (0, s]$. In case \ref{osc_proof_case_2} we have $0 < g(t) - 1 = g_2(t) - g_2(0) < 0$, where $g_2(t) = 1 - t$ for $t \in [0, 1]$, which is a contradiction. In the second case we have $g(t) = g_1(t)$ for all $t \in [0, s]$. However, now
\begin{equation*}
\infty > \# \left\{ u \in [0, s] \mid g_1(u) = 1 \right\} = \infty,
\end{equation*}
which is a contradiction.

This shows that the equation \eqref{ode_oscillation} does not have a solution. Since the moment function of any $L_2$-solution to \eqref{sde_oscillation_base} solves \eqref{ode_oscillation}, we conclude that the SDE \eqref{sde_oscillation_base} does not have an $L_2$-solution.

\end{proof}

\section*{Discussion} \label{section_discussion}
%
%
\revised{

In this article we restrict ourselves to the case where only the diffusion coefficient $\sigma$ has a discontinuity in the measure component. Regularity in the drift coefficient reduces the complexity of the equations we study and makes it easier to focus on understanding the effects caused by the discontinuity in the measure component of the diffusion term. However, one could consider a similar discontinuity also in the drift coefficient, that is, assume that 
\begin{equation*}
b(t, x, \mu, \alpha) = \sum_{j \in \mathcal{J}} \ind_{ \left\{ \alpha \in \cB_j \right\} } b_j(t, x, \mu).
\end{equation*}
One usual way to handle discontinuity in the drift coefficient is to use the Girsanov theorem -- see for example \cite{bauer_irregular_drift} or \cite{veretennikov79}, where the drift coefficient is discontinuous in the space variable --, but in our case this might change the law of $X_t$ and possibly $\E \left\Vert X_t - z \right\Vert^p$ so that this might not work. The more direct way would be to treat the switching of the drift coefficient similar to the switching of the diffusion coefficient and examine the behavior of the $L_p$-moments. Here the two switching regimes would interact with each other and this is a subject for future research and outside the scope of this article.

The standard assumptions \ref{standard_assumption_first}-\ref{standard_assumption_last} were chosen to ensure that the statement of Theorem \ref{th_strong_existence} holds in multiple dimensions. The linear growth condition \ref{standard_assumption_2} is used repeatedly in the proofs, and relaxing it would require to replace the Gronwall-type arguments with different and possibly more complicated arguments. The Lipschitz continuity assumption \ref{standard_assumption_1} is only used to obtain the existence and uniqueness of equations \eqref{eq_sde_initial_data}, so provided that the statement of Theorem \ref{th_strong_existence} would still hold, one could replace the Lipschitz continuity in the space variable with a weaker assumption.
}

\section*{Acknowledgement}
%
%
The author thanks Christel Geiss and Stefan Geiss for helpful discussions. The author is also grateful to the reviewers for their comments and suggestions to improve the article.

\bibliographystyle{plain}
%
%

\end{document}